\documentclass[preprint]{elsarticle}

\usepackage{lineno,hyperref}
\modulolinenumbers[5]

\usepackage{geometry}
\usepackage{pifont}
\usepackage{amscd}
\usepackage{amsmath,amsfonts,amssymb,amsthm}
\usepackage{latexsym}
\usepackage{stmaryrd}
\usepackage{overpic}
\usepackage{extarrows}
\usepackage{subfig}
\usepackage{multirow}
\usepackage{hyperref}
\journal{Elsevier}

\newtheorem{theorem}{Theorem}[section]
\newtheorem{lemma}[theorem]{Lemma}

\newdefinition{remark}{Remark}[section]

\numberwithin{equation}{section} 
\numberwithin{table}{section}
\numberwithin{figure}{section}

\begin{document}

\begin{frontmatter}



\title{ On the Stability and Accuracy of Partially and Fully
\\Implicit Schemes for Phase Field Modeling \tnoteref{This work is
supported in part by the U.S.  Department of Energy, Office of
Science, Office of Advanced Scientific Computing Research as part of
the Collaboratory on Mathematics for Mesoscopic Modeling of Materials
under contract number DE-SC0009249.}}
\tnotetext[Funding]{
This work is supported in part by the U.S.  Department of Energy,
Office of Science, Office of Advanced Scientific Computing Research as
part of the Collaboratory on Mathematics for Mesoscopic Modeling of
Materials under contract number DE-SC0009249.}

\author[psu]{Jinchao Xu\corref{cor1}}
\ead{xu@math.psu.edu}

\author[osu]{Yukun Li}
\ead{li.7907@osu.edu}

\author[psu]{Shuonan Wu}
\ead{sxw58@psu.edu}

\author[psu]{Arthur Bousquet}
\ead{akb5670@psu.edu}
  
\address[psu]{Department of Mathematics, Pennsylvania State University,
University Park, PA, 16802, USA}

\address[osu]{Department of Mathematics, The Ohio State University,
Columbus, OH, 43210, USA}

\cortext[cor1]{Corresponding author}

\begin{abstract}
  We study in this paper the accuracy and stability of partially and
  fully implicit schemes for phase field modeling.  Through
  theoretical and numerical analysis of Allen-Cahn and Cahn-Hillard
  models, we investigate the potential problems of using partially
  implicit schemes, demonstrate the importance of using fully implicit
  schemes and discuss the limitation of energy stability that
  are often used to evaluate the quality of a numerical scheme for
  phase-field modeling.  In particular, we make the following
  observations:
\begin{enumerate}
\item a convex splitting scheme (CSS in short) can be equivalent to
  some fully implicit scheme (FIS in short) with a much different time
  scaling and thus it may lack numerical accuracy;
\item most implicit schemes (in discussions) are energy-stable if the
   time-step size is sufficiently small;
 \item a traditionally known conditionally energy-stable scheme still
   possess an unconditionally energy-stable physical solution;
\item an unconditionally energy-stable scheme is not necessarily
better than a conditionally energy-stable scheme when the time step
size is not small enough;
\item a first-order FIS for the Allen-Cahn model can be
  devised so that the maximum principle will be valid on the discrete
  level and hence the discrete phase variable satisfies $|u_h(x)|\le
  1$ for all $x$ and, furthermore, the linearized discretized system
  can be effectively preconditioned by discrete Poisson operators.
\end{enumerate}
\end{abstract}


\end{frontmatter}

\section{Introduction} \label{sec:intro}
%%Partially implicit (or partially explicit) schemes, such as the convex
%%splitting schemes (CSS in short), are among the most popular numerical
%%schemes used in phase-field modeling.  
In this paper, we consider the following Allen-Cahn model
\cite{Allen_Cahn79}:
\begin{equation}\label{eq:AC}
\begin{aligned}
u_t - \Delta u + \frac{1}{\epsilon^2}f(u) &=0 \qquad \mbox{in } 
\Omega_T:=\Omega\times(0,T),\\
\frac{\partial u}{\partial n} &=0 \qquad \mbox{on }
\partial\Omega_T:=\partial\Omega\times(0,T),
\end{aligned}
\end{equation}
and the following Cahn-Hilliard model \cite{Cahn_Hilliard58}:
\begin{equation} \label{eq:CH}
\begin{aligned}
u_t-\Delta w &=0  \qquad \mbox{in } \Omega_T,\\
-\epsilon\Delta u +\frac{1}{\epsilon}f(u) &=w \qquad \mbox{in }
\Omega_T,\\
\frac{\partial u}{\partial n} =\frac{\partial w}{\partial n} &=0
\qquad \mbox{on } \partial\Omega_T.
\end{aligned}
\end{equation}
The initial condition is set as $u|_{t=0} = u_0$. Here, $T$ is the end
time, $\Omega\subset \mathbb{R}^d \,(d=2,3)$ is
a bounded domain and $f=F'$ for some double well potential $F$ which, in
this paper, is taken to be the following polynomial:
\begin{equation}\label{double-well} F(u)=\frac{1}{4}(u^2-1)^2.
\end{equation}

In recent years, there have been a lot of studies in the literature on
the modeling aspects and their numerical solutions for both Allen-Cahn
and Cahn-Hilliard equations.  For the modeling aspects, we refer to
\cite{Allen_Cahn79, cahn1996limiting, Cahn_Hilliard58, novick1998cahn,
  xinfu1994spectrums, Chen96, Evans_Soner_Souganidis92,
  boyer2014hierarchy, wu2017multiphase}.  In this paper, we will focus
on the numerical schemes for both these equations.  Among the various
different schemes studied in the literature, a special class of
partially implicit schemes, known as convex splitting schemes, appears
to be most popular, c.f.~\cite{guillen2014second, guan2014second,
  feng2015long, shen2011spectral, shen2010numerical, yang2009error,
  Feng_Li15} for the Allen-Cahn equation and \cite{guillen2014second,
  aristotelous2013mixed, feng2015long, shen2010numerical,
  shen2011spectral, shen2010energy, guan2014second, feng2015analysis,
  Eyre98} for the Cahn-Hilliard model.  The popularity of the CSS is
due to, among others, its two advantages: (1) a typical CSS is
unconditionally energy-stable without any stringent restriction
pertaining to the time step; (2) the resulting nonlinear numerical
system can be easily solved (e.g. Newton iteration is guaranteed to
converge regardless of the initial guess).  In comparison, a standard
fully implicit scheme is only energy-stable when the time step size is
sufficiently small.

It is against the conventional wisdom that a partially implicit scheme
such as the convex splitting scheme has a better stability property
than a fully implicit scheme.  One main goal of this paper is to
understand this unusual phenomenon.  For the Allen-Cahn model, we
prove that the standard first-order CSS is exactly the same as the
standard first-order FIS but with a (much) smaller time step size and
as a result, it would provide an approximation to the original
solution of the Allen-Cahn model at a delayed time (although the
magnitude of the delay is reduced when the time step size is
reduced). Such a time delay is also observed for other partially
implicit schemes when time step size is not sufficient small. For the
Cahn-Hilliard model, we prove that the standard CSS is exactly the
same as the standard FIS for a different model that is a (nontrivial)
perturbation of the original Cahn-Hilliard model.  This at least
explains theoretically why a CSS has a better stability property than
a FIS does since a CSS is actually a FIS with a very small time-step
size. In addition, we argue that such a gain of stability is at the
expense of a possible loss of accuracy.

Given the aforementioned equivalences between CSS and FIS and the
popularity of CSS in the literature, the value of FIS with a seemingly
stringent time-step constraint (which, again, are equivalent to CSS
without any time-step constraint) should be re-examined. 
Indeed, the importance of using fully implicit schemes for the phase
field simulations has been addressed in the existing literature, e.g.
\cite{du1991numerical, feng2003numerical, feng2007analysis,
rosam2007fully, GBH08, shen2011spectral, graser2013time,
guillen2014second, Feng_Li15, feng2015analysis, li2015numerical,
wu2017multiphase}.  In this paper, we further study
three families of new algorithms for FIS. First, we
revisit the standard fully implicit scheme by extending it to a energy
minimization problem at each time step. The minimization problem,
however, admits a non-convex discrete energy when the time step size
is not sufficiently small. Furthermore, we will be able to prove, rather
straightforwardly, that the global minimizer satisfies the
unconditional energy-stability, which is a natural property for linear
systems and the desired property for the nonlinear systems like the
Allen-Cahn or the Cahn-Hilliard equations. The results given by the
energy minimization problem is quite different from those given by the
standard fully implicit scheme. More precisely, instead of the severe
restriction pertaining to the time step size, the energy minimization
problem gives a good approximation to the physical solution only when
the discretization error in time is controlled. Moreover, with the
energy minimization problem, various minimization solvers (e.g.
L-BFGS \cite{nocedal1980updating, byrd1995limited}) can be efficiently
applied. This may lead a promising direction to the design of accurate
and efficient numerical schemes for phase field modeling.

Secondly, we propose a modification of a typical FIS for the
Allen-Cahn so that the maximum principle will be valid on the discrete
level.  Thirdly, for this modified FIS scheme, we rigorously show
that, under the appropriate time-step size constraint, the
linearization of such a modified FIS can be uniformly preconditioned
by a Poisson-like operator.

Second-order partially implicit schemes have also been designed in the
literature with the same purpose of allowing large time step size as
the first-order partially implicit schemes. But similar to the
standard CSS, the time delay happens with large time step
size. Actually, the second-order CSS (cf. \cite{guillen2014second,
shen2011spectral, shen2010numerical, yang2009error}) can also be
viewed as the modified Crank-Nicolson scheme \cite{du1991numerical,
shen2010numerical, condette2011spectral} on the artificially
convexified model. Further, we demonstrate that, through numerical
experiments with the modified Crank-Nicolson scheme, an
unconditionally energy stable scheme is not necessarily better than a
conditionally energy stable scheme.

The rest of paper is organized as follows.  In \S \ref{sec:1st-order},
we focus on the first-order schemes. We study the convexity of the
fully implicit scheme, prove that a typical first-order CSS is exactly
equivalent to some first-order FIS.  We also introduce the energy
minimization version of some first-order FIS, and show that the convex
splitting schemes can be viewed as artificial convexity schemes.  In
\S \ref{sec:modified-FIS}, we propose a modified FIS (or CSS) that
satisfies maximum principle on the discrete level and further prove
that the modified scheme can be preconditioned by a Poisson-like
operator. In \S \ref{sec:2nd-order}, we discuss the second-order
schemes. We study a modified Crank-Nicolson scheme and its convex
splitting version, compare the modified Crank-Nicolson scheme and some
other second-order partially implicit schemes.  Finally, in \S
\ref{sec:concluding}, we give some concluding remarks. 

%%%%%%%%%%%%%%%%%%%%%%%%%%%%%%%%%%%%%%%%%%%%%%%%%%%%%%%%%%
%% First-order schemes 
%%%%%%%%%%%%%%%%%%%%%%%%%%%%%%%%%%%%%%%%%%%%%%%%%%%%%%%%%%
\section{First-order schemes} \label{sec:1st-order}
First, we introduce some notation. Let $\mathcal{T}_h$ be a
shape-regular (which may not be quasi-uniform) triangulation of
$\Omega \subset \mathbb{R}^d \ (d=2,3)$. The nodes of $\mathcal{T}_h$
is denoted by $\mathcal{N}_h$. $K$ represents each element and
$\overline{\Omega}=\bigcup_{K\in\mathcal{T}_h} \overline{K}$.  Let
$h_K$ denote the diameter of $K\in \mathcal{T}_h$ and
$h:=\mbox{max}\{h_K; K\in\mathcal{T}_h\}$.  Define the finite element
space $V_h$ by
\begin{equation} \label{eq:Pr}
V_h = \bigl\{v_h\in C(\overline{\Omega}): v_h|_K=P_r(K)\bigr\},
\end{equation}
where $P_r(K)$ denotes the set of all polynomials whose degrees do not
exceed a given positive integer $r$ on $K$.  The $L^2$-inner product
over the domain $\Omega$ is denoted by $(\cdot, \cdot)$. For the time
discretization, let $k_n$ denote the time step size on $n$-th step and
$t_n := \sum_{i=1}^n k_i$.

\subsection{Fully implicit schemes and their convexity and energy
stability properties} \label{sec:FIS}
A standard first-order fully implicit scheme to problem \eqref{eq:AC}
(FIS in short) is defined by seeking $u_h^n\in V_h$ for $n=1,2,\cdots
$, such that
\begin{equation}\label{eq:FIS-AC}
(\frac{u_h^{n}-u_h^{n-1}}{k_n},v_h)+(\nabla u_h^{n},\nabla v_h) +
\frac{1}{\epsilon^2}(f(u_h^{n}),v_h) = 0 \qquad\forall v_h\in V_h.
\end{equation}
 
A standard first-order FIS to problem \eqref{eq:CH} is defined by
seeking $u_h^n \in V_h$ and $w_h^n \in V_h$ for $n=1,2,\cdots$, such
that 
\begin{equation} \label{eq:FIS-CH}
\begin{aligned}
(\frac{u_h^{n}-u_h^{n-1}}{k_n},\eta_h)+(\nabla w_h^{n},\nabla \eta_h) &=0
\qquad \forall \,\eta_h\in V_h, \\
\epsilon (\nabla u_h^{n},\nabla v_h) + \frac{1}{\epsilon}((u_h^n)^3 -
u_h^n,v_h)-(w_h^{n},v_h) &=0 \qquad \forall\,
  v_h\in V_h. 
\end{aligned}
\end{equation}
 
Following \cite{Evans_Soner_Souganidis92,Ilmanen93}, the
Allen-Cahn equation \eqref{eq:AC} can be interpreted as the
$L^2$-gradient flow for the free-energy functional, namely
\begin{equation}\label{eq:energy-AC}
\begin{aligned}
J_\epsilon^{\rm AC}(v) &:= 
\int_\Omega \Bigl( \frac12 |\nabla v|^2
+ \frac{1}{\epsilon^2} F(v) \Bigr)\, dx, \\
\frac{d}{dt}J_\epsilon^{\rm AC}(u(t)) &= \bigl(-\Delta
u+\frac{1}{\epsilon^2}f(u),u_t\bigr)_{L^2(\Omega)}=-\|u_t\|_{L^2(\Omega)}^2\leq0.
\end{aligned}
\end{equation}

Following \cite{Alikakos_Bates_Chen94,Chen96,Pego89}, the
Cahn-Hilliard equations \eqref{eq:CH} can be interpreted as the
$H^{-1}$-gradient flow for the free-energy functional, namely
\begin{equation}\label{eq:energy-CH}
\begin{aligned}
J_\epsilon^{\rm CH}(v) &:= \int_\Omega \Bigl( \frac\epsilon2 |\nabla v|^2+
\frac{1}{\epsilon} F(v) \Bigr)\, dx, \\
\frac{d}{dt}J_\epsilon^{\rm CH}(u(t)) &= \bigl(\Delta(\epsilon\Delta
u-\frac{1}{\epsilon}f(u)),u_t\bigr)_{H^{-1}(\Omega)}
=-\|u_t\|_{H^{-1}(\Omega)}^2\leq 0.
\end{aligned}
\end{equation}
Therefore, we say that a discretization scheme such as \eqref{eq:FIS-AC} or
\eqref{eq:FIS-CH} is energy-stable if 
\begin{equation} \label{AC-CH-stability}
  J_\epsilon^{\rm AC}(u^n_h)\le   J_\epsilon^{\rm AC}(u^{n-1}_h) \quad
  \mbox{or} \quad 
  J_\epsilon^{\rm CH}(u^n_h)\le   J_\epsilon^{\rm CH}(u^{n-1}_h) \qquad
  n=1,2,\ldots
\end{equation}

We would like to point out that the concept of energy-stability for
the nonlinear schemes such as \eqref{eq:FIS-AC} or
\eqref{eq:FIS-CH} is different from the standard concept
of stability for linear schemes.  For most linear systems (e.g.
heat equation), a fully implicit scheme is usually unconditionally
stable.  But for nonlinear systems, fully implicit schemes such as
\eqref{eq:FIS-AC} or \eqref{eq:FIS-CH} are only conditionally
energy-stable, namely they are only energy-stable when the time-step
size $k_n$ is appropriately small.  This is well-known fact in the
phase-field literature (cf. \cite{feng2003numerical, graser2013time}).
For completeness, we will study this energy-stability property through
the study of the convexity of the relevant schemes. Further, we extend
the standard schemes to the energy minimization versions at each
time step, which seem to have better numerically performance. 

\subsubsection{Convexity of fully implicit schemes for the Allen-Cahn
equation}
In this section, we next study the convexity property of the FIS of
the Allen-Cahn and Cahn-Hilliard equations.  Consider the Allen-Cahn
equation, in view of \eqref{eq:energy-AC}, we define the following
discrete energy
\begin{equation} \label{eq:FIS-AC-energy} 
E_n^{\rm AC}(u_h ;u_h^{n-1}) 
=J_\epsilon^{\rm AC}(u_h)+\frac{1}{2k_n}\int_{\Omega}(u_h - u_h^{n-1})^2dx.
\end{equation}
We also extend the standard first-order fully implicit  scheme to the
following energy minimization problem:  
\begin{equation} \label{equ:FIS-AC-energy-min}
u_h^n = \underset{u_h\in V_h}{\mathrm{argmin}} E_n^{\rm AC}(u_h;u_h^{n-1}).
\end{equation}

\begin{theorem}\label{thm:convexity-FIS-AC}
We have 
\begin{enumerate}
\item Under the condition that $k_n\leq \epsilon^2$, $E_n^{\rm
AC}(\cdot;u_h^{n-1})$ is strictly convex on $V_h$.
\item The solution of \eqref{eq:FIS-AC} satisfies $(E_n^{\rm
AC})'(u_h^n; u_h^{n-1})(v_h) = 0$.
\item The following discrete energy law holds for
\eqref{equ:FIS-AC-energy-min}
\begin{align}\label{eq:FIS-AC-energy-law} 
J_\epsilon^{\rm AC}(u_h^n) + \frac{1}{2k_n}\|u_h^n -
u_h^{n-1}\|_{L^2(\Omega)}^2 \leq J_\epsilon^{\rm AC}(u_h^{n-1}).
\end{align}
\end{enumerate}
\end{theorem}

\begin{proof}
Taking the second derivative of $E_n^{\rm AC}(\cdot;u_h^{n-1})$, we
get for any $v_h \in V_h$,
\begin{equation}\label{eq:FIS-AC-frechet}
(E_n^{\rm AC})''(u_h;u_h^{n-1})(v_h,v_h) =
\frac{3}{\epsilon^2}\int_{\Omega}u_h^2v_h^2dx +
\int_{\Omega}(\frac{1}{k_n}-\frac{1}{\epsilon^2})v_h^2dx + \|\nabla
v_h\|_{L^2(\Omega)}^2.
\end{equation}
When $k_n\leq \epsilon^2$, $(E_n^{\rm AC})''(u_h;u_h^{n-1})(v_h,v_h)>0$ when
$v_h\neq 0$, which means $E(\cdot;u_h^{n-1})$ is strictly convex on
$V_h$.
A direct calculation shows that \eqref{eq:FIS-AC} satisfies $(E_n^{\rm
AC})^{\prime}(u_h^n;u_h^{n-1})(v_h) = 0$, and the following coercivity
condition holds:
\begin{equation}\label{eq:FIS-AC-energy-boundedness}
E_n^{\rm AC}(u_h;u_h^{n-1}) \geq M_1\|u_h\|_{H^1(\Omega)}^2 - M_2,
\end{equation}
where $M_1$ and $M_2$ are positive constants that depend on
$\epsilon$.  Then the unique solvability of \eqref{eq:FIS-AC} follows
from \cite{ciarlet1989introduction} and
\eqref{eq:FIS-AC-energy-boundedness}. Moreover, for the global
minimizer of \eqref{equ:FIS-AC-energy-min}, we have  
$$ 
J_\epsilon^{\rm AC}(u_h^n) + \frac{1}{2k_n}\|u_h^n -
u_h^{n-1}\|_{L^2(\Omega)}^2
= E_n^{\rm AC}(u_h^n;u_h^{n-1}) \leq E_n^{\rm AC}(u_h^{n-1};u_h^{n-1}) 
= J_\epsilon^{\rm AC}(u_h^{n-1}).
$$ 
Then we finish the proof.
\end{proof}

In view of Theorem \ref{thm:convexity-FIS-AC}, let us introduce the
terminology of {\it convex scheme}.  We say that a scheme is {\it
convex} if it is equivalent to the minimization of a convex
functional.  Thus \eqref{eq:FIS-AC} is a convex scheme under the
condition $k_n \leq \epsilon^2$, under which the first-order FIS
\eqref{eq:FIS-AC} is equivalent to the energy minimization version
\eqref{equ:FIS-AC-energy-min}. When $k_n > \epsilon^2$, the
$E_n^{\rm AC}(\cdot;u_h^{n-1})$ may not be convex, hence the standard
Newton's method for \eqref{eq:FIS-AC} may fail in this case.  Thus,
generally speaking, the scheme \eqref{equ:FIS-AC-energy-min} calls for
the global minimization solver.

\subsubsection{Convexity of fully implicit scheme for the
Cahn-Hilliard equation}

Define the discrete Laplace operator $\Delta_h: {V}_h\mapsto
{V}_h$ as follows: Given $v_h \in {V}_h$, let $\Delta_h v_h \in{V}_h$
such that
\begin{equation} \label{eq:discrete-Laplace}
(\Delta_h v_h, w_h)=-(\nabla v_h,\nabla w_h) \qquad \forall\, w_h\in V_h.
\end{equation}
Let $L^2_0$ denote the collection of functions in $L^2(\Omega)$ with
zero mean, and let $\mathring{V}_h:=V_h\cap L^2_0$. Taking $w_h = 1$
in \eqref{eq:discrete-Laplace}, we know that $\text{Range}(\Delta_h) \subset
\mathring{V}_h$. Further, the well-posedness of the Poisson problem
with Neumann boundary condition on $\mathring{V}_h$ implies that
$\text{Range}(\Delta_h) = \mathring{V}_h$. Therefore,
$\Delta_h|_{\mathring{V}_h}: \mathring{V}_h \mapsto \mathring{V}_h$
is an isomorphism, then $\Delta_h^{-1} :=
(\Delta_h|_{\mathring{V}_h})^{-1}: \mathring{V}_h \mapsto
\mathring{V}_h$ is well-defined.

Consider the Cahn-Hilliard equations, in view of \eqref{eq:energy-CH},
we define the discrete energy
\begin{equation} \label{eq:FIS-CH-energy} 
E_n^{\rm CH}(\theta_h;u_h^{n-1}) =
J_\epsilon^{\rm CH}(u_h^{n-1}+\theta_h) +
\frac{1}{2k_n}\|\nabla\Delta_h^{-1}\theta_h\|_{L^2(\Omega)}^2 \qquad
\theta_h\in \mathring{V}_h. 
\end{equation}
Then, the energy minimization version of the first-order FIS for the
Cahn-Hilliard equations is shown to be  
\begin{equation} \label{equ:FIS-CH-energy-min}
\theta_h^n = \underset{\theta_h \in \mathring{V}_h}{\mathrm{argmin}} E_n^{\rm
  CH}(\theta_h; u_h^{n-1}), \qquad u_h^{n} = u_h^{n-1} + \theta_h^n. 
\end{equation}

\begin{theorem}\label{thm:convexity-FIS-CH}
We have  
\begin{enumerate}
\item Under the condition that $k\leq 4\epsilon^3$, $E_n^{\rm CH}(\cdot; u_h^{n-1})$ is convex on $\mathring{V}_h$.
\item The solution of \eqref{eq:FIS-CH} satisfies
$u_h^{n}=u_h^{n-1}+\theta_h$, with $(E_n^{\rm CH})'(\theta_h;
u_h^{n-1})(v_h) = 0$.
\item The following energy law holds for \eqref{equ:FIS-CH-energy-min} 
\begin{equation} \label{eq:CH-energy}
J_\epsilon^{\rm CH}(u_h^{n}) +
\frac{1}{2k_n}\|\nabla\Delta_h^{-1}(u_h^{n}-u_h^{n-1})\|_{L^2(\Omega)}^2
\le J_\epsilon^{\rm CH}(u_h^{n-1}).
\end{equation}
\end{enumerate}

\end{theorem}
\begin{proof}
For any $\theta_h, \eta_h\in \mathring{V}_h$, we have  
$$
 (E_n^{\rm CH})''(\theta_h;u_h^{n-1})(\eta_h, \eta_h) = \frac{1}{{\epsilon}}\int_{\Omega}\bigl(3(u_h^{n-1}+\theta_h)^2-1\bigr)\eta_h^2 ~dx +
\frac{1}{k_n}\|\nabla\Delta_h^{-1}\eta_h\|_{L^2(\Omega)}^2 +
\epsilon\|\nabla \eta_h\|_{L^2(\Omega)}^2.
$$
Using Schwarz's inequality, we have
$$
\frac{1}{\epsilon}\|\eta_h\|_{L^2(\Omega)}^2  \leq \frac{1}{\epsilon}
(\Delta_h^{-1}\eta_h,\eta_h)^{1/2} (\Delta_h\eta_h,\eta_h)^{1/2}\leq
\frac{1}{4\epsilon^3}\|\nabla\Delta_h^{-1}\eta_h\|_{L^2(\Omega)}^2 +
\epsilon\|\nabla \eta_h\|_{L^2(\Omega)}^2.
$$
When $k_n\leq 4\epsilon^3$, 
\begin{align}\label{eq:FIS-CH-convex} 
(E_n^{\rm CH})''(\theta_h;u_h^{n-1})(\eta_h,\eta_h) \geq
\frac{1}{{\epsilon}}\int_{\Omega}3(u_h^{n-1}+\theta_h)^2 \eta_h^2~dx
+(\frac{1}{k_n}-\frac{1}{4\epsilon^3})\|\nabla\Delta_h^{-1}\eta_h\|_{L^2(\Omega)}^2\geq0,
\end{align}
where the strict inequality holds when $\eta_h\neq 0$.
This means that $(E_n^{\rm CH})(\cdot;u_h^{n-1})$ is strictly convex on
$\mathring{V}_h$. 

Now, taking $\eta_h = 1$ in \eqref{eq:FIS-CH}, we have $u_h^n \in
u_h^{n-1} + \mathring{V}_h$. Let $v_h = 1$ in \eqref{eq:FIS-CH}, we
have $\int_{\Omega}w_h^ndx =
\frac{1}{\epsilon}\int_{\Omega}f(u_h^n)dx$.  Then, the first equation
of \eqref{eq:FIS-CH} is equivalent to 
$$ 
w_h^n = \frac{1}{k_n}\Delta_h^{-1}(u_h^n -
    u_h^{n-1})+\frac{1}{\epsilon|\Omega|}\int_{\Omega}f(u_h^n)dx. 
$$
Therefore, \eqref{eq:FIS-CH} is shown to be 
\begin{equation} \label{eq:FIS-CH-u}
\epsilon(\nabla u_h^n, \nabla v_h) + \frac{1}{\epsilon}((I-Q_0)f(u_h^n), v_h)
- \frac{1}{k_n}(\Delta_h^{-1}(u_h^n-u_h^{n-1}), v_h) 
= 0 \qquad\forall~v_h \in V_h.
\end{equation}
where $Q_0: L^2(\Omega)\mapsto \mathbb R$ is the $L^2$ projection,
namely $ Q_0v=\frac{1}{|\Omega|}\int_\Omega vdx$.  Let
$\theta_h=u_h^n-u_h^{n-1} \in \mathring{V}_h$. Note that
$Q_0\theta_h=Q_0\Delta^{-1}\theta_h=0$, we can then write
\eqref{eq:FIS-CH-u} as
\begin{equation*}
\epsilon(\nabla (u_h^{n-1}+\theta_h), \nabla (I-Q_0)v_h) +
\frac{1}{\epsilon}(f(u_h^{n-1} + \theta_h), (I-Q_0)v_h)
- \frac{1}{k_n}(\Delta_h^{-1}\theta_h, (I-Q_0)v_h) 
= 0 \quad\forall~v_h \in V_h.
\end{equation*}
This means that
\begin{equation*}
\epsilon(\nabla (u_h^{n-1}+\theta_h), \nabla v_h) +
\frac{1}{\epsilon}(f(u_h^{n-1}+\theta_h), v_h) -
\frac{1}{k_n}(\Delta_h^{-1}\theta_h, v_h) = 0 \qquad\forall~v_h
\in \mathring V_h,
\end{equation*}
which can be recast as $(E_n^{\rm CH})'(\theta_h; u_h^{n-1})(v_h) =
0$. The unique solvability and energy stability \eqref{eq:CH-energy}
then follows from the similar argument in Theorem
\ref{thm:convexity-FIS-AC}. 
\end{proof}

\subsection{Convex splitting schemes and their equivalence to fully
  implicit schemes} \label{sec:CSS-FIS}
As we have seen before, convexity is a very desirable property of the
discretize scheme and fully implicit scheme is only convex when $k$ is
sufficiently small. When $k$ is not sufficiently small, the
non-convexity of the discrete scheme comes from the fact that the
potential function $F$ in \eqref{double-well} is not convex.
The convex splitting scheme (CSS in short) stems from splitting  the
non-convex potential function $F$ given by \eqref{double-well} into
the difference between two convex functions:
\begin{equation} \label{convex-splittingF}
F(u)=F_+(u)-F_-(u), \quad \mbox{with} \quad  F_+(u)=\frac{1}{4}(u^4+1), \quad F_-(u)=\frac{1}{2}u^2.
\end{equation}

\subsubsection{A convex splitting scheme for the Allen-Cahn model}
\label{subsec:CSS-AC}
In view of Theorem \ref{thm:convexity-FIS-AC}, a CSS be obtained by
making the non-convex part, namely $-F_-(\cdot)$ in
\eqref{convex-splittingF}, explicit in some way, and it can be
characterized by the minimization of a convex functional:
\begin{equation}
  \label{CSS-min}
u_h^n = \underset{u_h\in V_h}{\mathrm{argmin}} 
\bigg\{
\int_\Omega \Bigl( \frac12 |\nabla u_h|^2
+ \frac{1}{\epsilon^2} [F_+(u_h)-\hat F_-(u_h;u_h^{n-1})] \Bigr)\, dx
+
\frac{1}{2k_n}\int_{\Omega}(u_h- u_h^{n-1})^2dx
 \bigg\},
\end{equation}
where $\hat F_-(u_h;u_h^{n-1}) $ is the linearization of $F_-(\cdot)$ at
$u_h^{n-1}$, that is, $\hat F_-(u_h;u_h^{n-1})
=F_-(u_h^{n-1})+F_-^{\prime}(u_h^{n-1}) (u_h-u_h^{n-1})$.

The variational formulation of \eqref{CSS-min} is the following
well-known CSS: Find $u_h^n\in V_h$ for $n=1,2,\cdots$, such that
\begin{equation}\label{css}
(\frac{u_h^{n}-u_h^{n-1}}{k_n},v_h)+(\nabla u_h^{n},\nabla
v_h)+\frac{1}{\epsilon^2}((u_h^{n})^3-u_h^{n-1},v_h)=0
\qquad\forall v_h\in V_h.
\end{equation}

%The most attractive property of this CSS is stated below.
\begin{theorem} \cite{Eyre98}
The CSS scheme \eqref{css} is unconditionally energy stable.   
\end{theorem}
At the first glance, the above result looks incredibly remarkable. As
we have seen above, even a fully implicit scheme can not be
unconditionally energy-stable, but as a partially implicit (or
explicit) scheme, CSS is unconditionally energy-stable.  Although, as
we discussed before, we can not quite relate the energy-stability in a
nonlinear scheme to the standard stability concept in a standard
linear scheme,  it is quite incredible that a partially implicit (or
explicit) scheme is actually more stable than a fully implicit scheme!

This remarkable phenomenon can be explained by the following result. 

\begin{theorem}\label{thm:css-fis}
The CSS \eqref{css} can be recast as the FIS \eqref{eq:FIS-AC}
with different time step size:
\begin{equation} \label{kprime}
k_n' = \frac{\epsilon^2}{k_n+\epsilon^2}k_n.  
\end{equation}
%%\begin{equation} \label{convex-recast}
%%\bigl(\frac{u_h^n - u_h^{n-1}}{k'}, v_h\bigr) + (\nabla u_h^n, v_h) +
%%\frac{1}{\epsilon^2} (f(u_h^n), v_h) = 0 \qquad \forall v_h\in V_h,
%%\end{equation}
\end{theorem}
\begin{proof}
We write that 
$$
(u_h^n)^3 - u_h^{n-1} = f(u_h^n) + (u_h^n - u_h^{n-1}).
$$
Substituting the above identity into \eqref{css} and regrouping the
term involving $u_h^n - u_h^{n-1}$, we obtain
\begin{align}\label{css_1}
\bigl((\frac{1}{k_n}+\frac{1}{\epsilon^2})(u_h^{n}-u_h^{n-1}),v_h\bigr)+(\nabla u_h^{n},\nabla
v_h)+\frac{1}{\epsilon^2}(f(u_h^{n}),v_h)=0
\qquad\forall v_h\in V_h, 
\end{align}
which is exactly the FIS with time step size \eqref{kprime}. 
\end{proof}

By comparing the condition for the time step size in Theorem
\ref{thm:convexity-FIS-AC} and \eqref{kprime}, the resulting time-step
constraint \eqref{kprime} in the CSS is actually more stringent 
to assure the convexity of the original FIS, as $k_n' < \epsilon^2$ for
any $\epsilon>0$. This also explains why the CSS is always
energy-stable thanks to the Theorem \ref{thm:convexity-FIS-AC}.

\begin{remark}\label{remark:delay}
We now make some remark on the implication of Theorem \ref{thm:css-fis}.
Let $u_h^{\rm FIS}(t_n)$ be the solution to \eqref{eq:FIS-AC} and
$u_h^{{\rm CSS}}(t_n)$ be the solution to \eqref{css}.  Then by
Theorem \ref{thm:css-fis}, we have
\begin{equation} \label{delay}
u_h^{{\rm CSS}}(t_n)=u_h^{{\rm FIS}}(\delta_n t_n), \quad \mbox{with}
\quad \delta_n=\frac{\epsilon^2}{k_n+\epsilon^2}. 
\end{equation}
Here, $\delta_n$ can be regarded as a delaying factor.  A larger time
step size $k_n$, which gives a smaller $\delta_n$, leads to a more
significant time-delay.  Even for a very small $k_n$, such a delay is
not negligible.  For example, if $k_n=\epsilon^2$, we have $\delta_n=1/2$.
Thus, $u_h^{{\rm CSS}}(t_n)=u_h^{{\rm FIS}}(\frac{t_n}{2})$.

Because of such a delay, it is expected and also numerically verified
that, quantitatively speaking, the CSS may have a reduced accuracy
although it gives qualitatively correct answer.  Furthermore such a
delay will diminish as $k_n\to 0$ since $ \lim_{k_n\to 0}\delta_n=1$. 

In summary, we conclude that the CSS has a special property that may
be known as ``delayed convergence'' in the following sense:
\begin{enumerate}
\item The CSS scheme is expected to eventually converge to the
  exact solution of the originally Allen-Cahn equation as $k_n\to 0$.
\item But for any given time step size $k_n$, the CSS would approximate 
better the exact solution at a delayed time. 
\end{enumerate}
\end{remark}

\paragraph{Test 1} \label{test1} 
In this test, the square domain $\Omega =(-1,1)^2$ is used to
investigate the performance of different numerical schemes, and the
initial condition is chosen as
\begin{align} \label{eq:initial-smooth}
u_0 = \tanh\bigl(\frac{d_0(x)}{\sqrt{2}\epsilon}\bigr). 
\end{align}
Here, $d_0(x)$ is the signed distance function from $x$ to the initial
curve $\Gamma_0: x^2+y^2 = 0.6^2$, i.e., $d_0(x) =
\sqrt{x^2+y^2}-0.6$.  Figure \ref{fig:AC-diff-h} and
\ref{fig:AC-diff-eps} displays the evolution of the radius with
respect to time. The singularity happens at $t = 0.18$, which is the
disappearing time. 

The numerical solutions of FIS and CSS with different $h$'s are
plotted in Figure \ref{fig:AC-diff-h}. When decreasing $h$, the FIS
approximates the exact solution well, while the CSS does not. 
The similar phenomenon happens with different $\epsilon$'s, as shown
in Figure \ref{fig:AC-diff-eps}.

\begin{figure}[!htbp]
\centering 
\captionsetup{justification=centering}
\subfloat[FIS]{
  \includegraphics[width=0.35\textwidth]{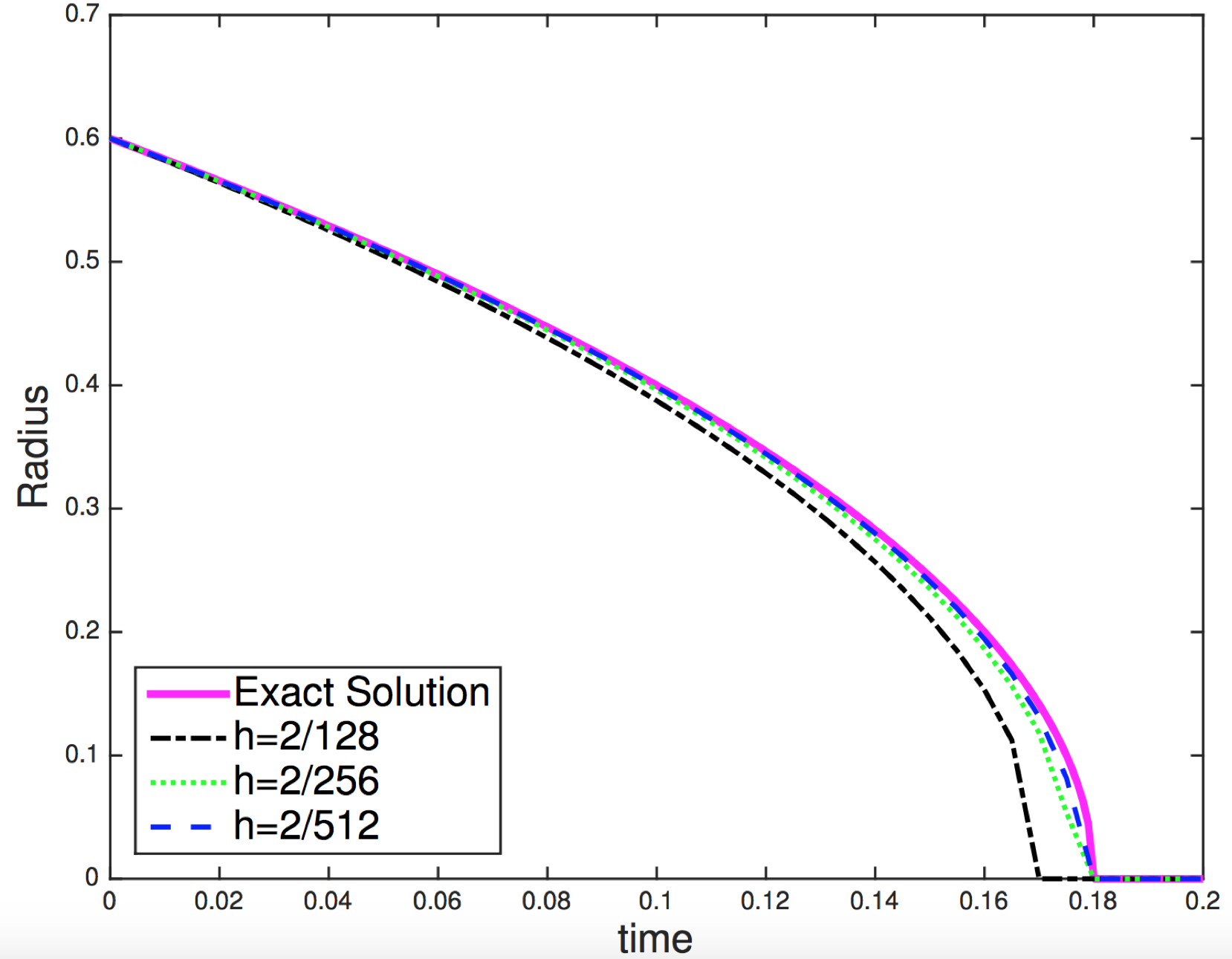}
  \label{fig:AC-diff-h-FIS}
}%
\subfloat[CSS]{
  \includegraphics[width=0.35\textwidth]{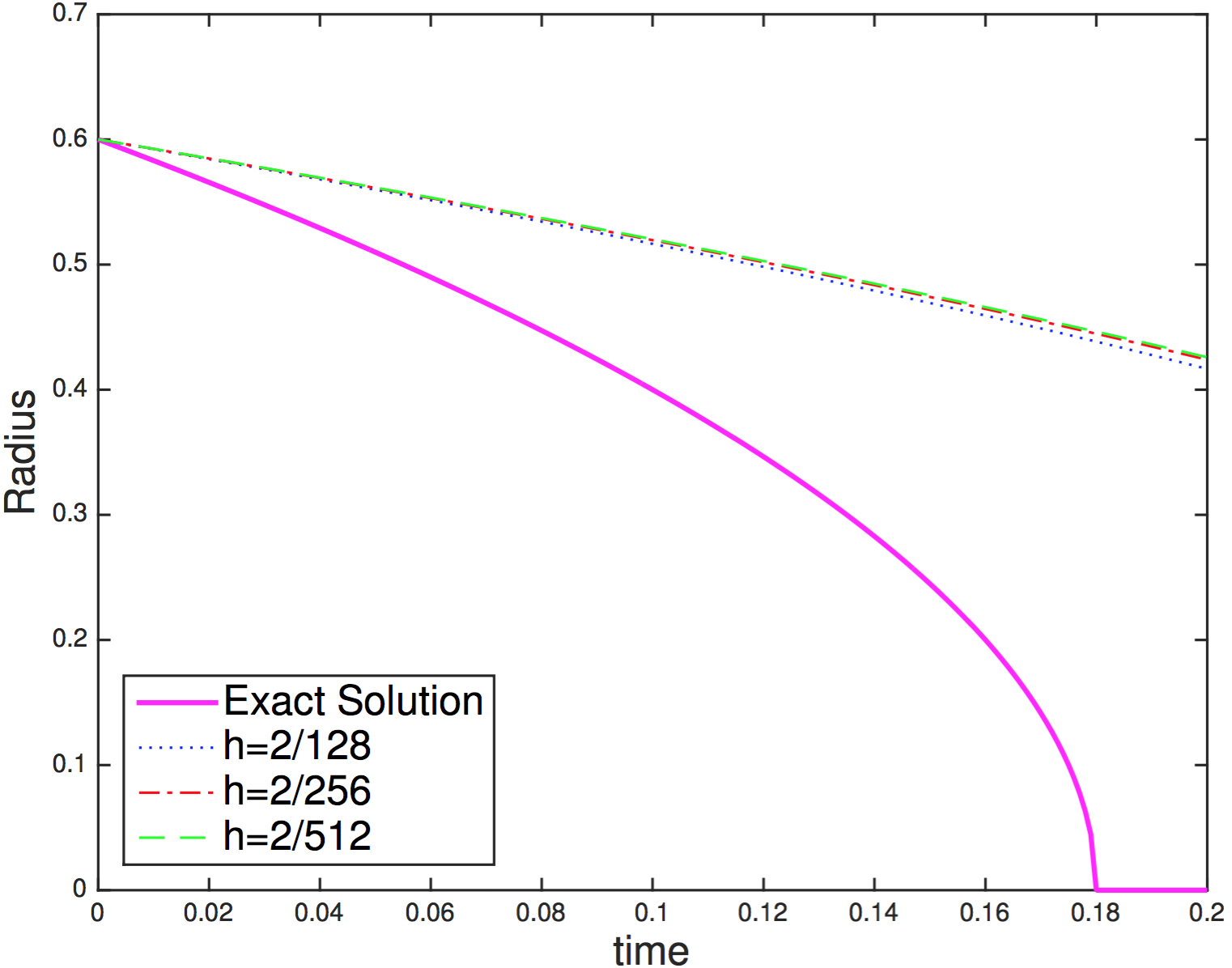}
  \label{fig:AC-diff-h-CCS}
}
\caption{Allen-Cahn equation: FIS and CSS with $\epsilon=0.02,
  k_n=0.0005$ and different $h$'s.} 
  \label{fig:AC-diff-h}
\end{figure}

\begin{figure}[!htbp]
\centering 
\captionsetup{justification=centering}
\subfloat[FIS]{
  \includegraphics[width=0.35\textwidth]{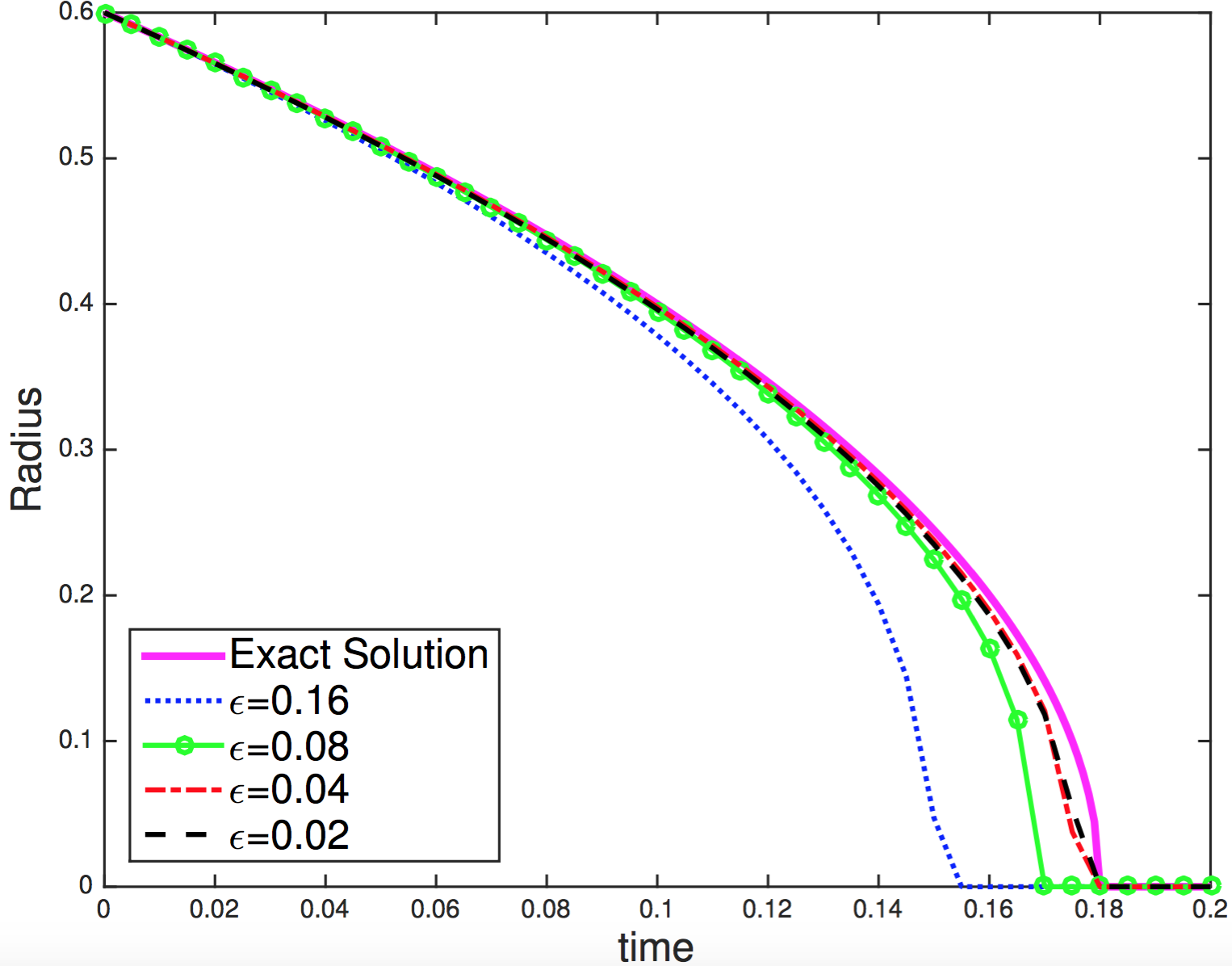}
  \label{fig:AC-diff-eps-FIS}
}%
\subfloat[CSS]{
  \includegraphics[width=0.35\textwidth]{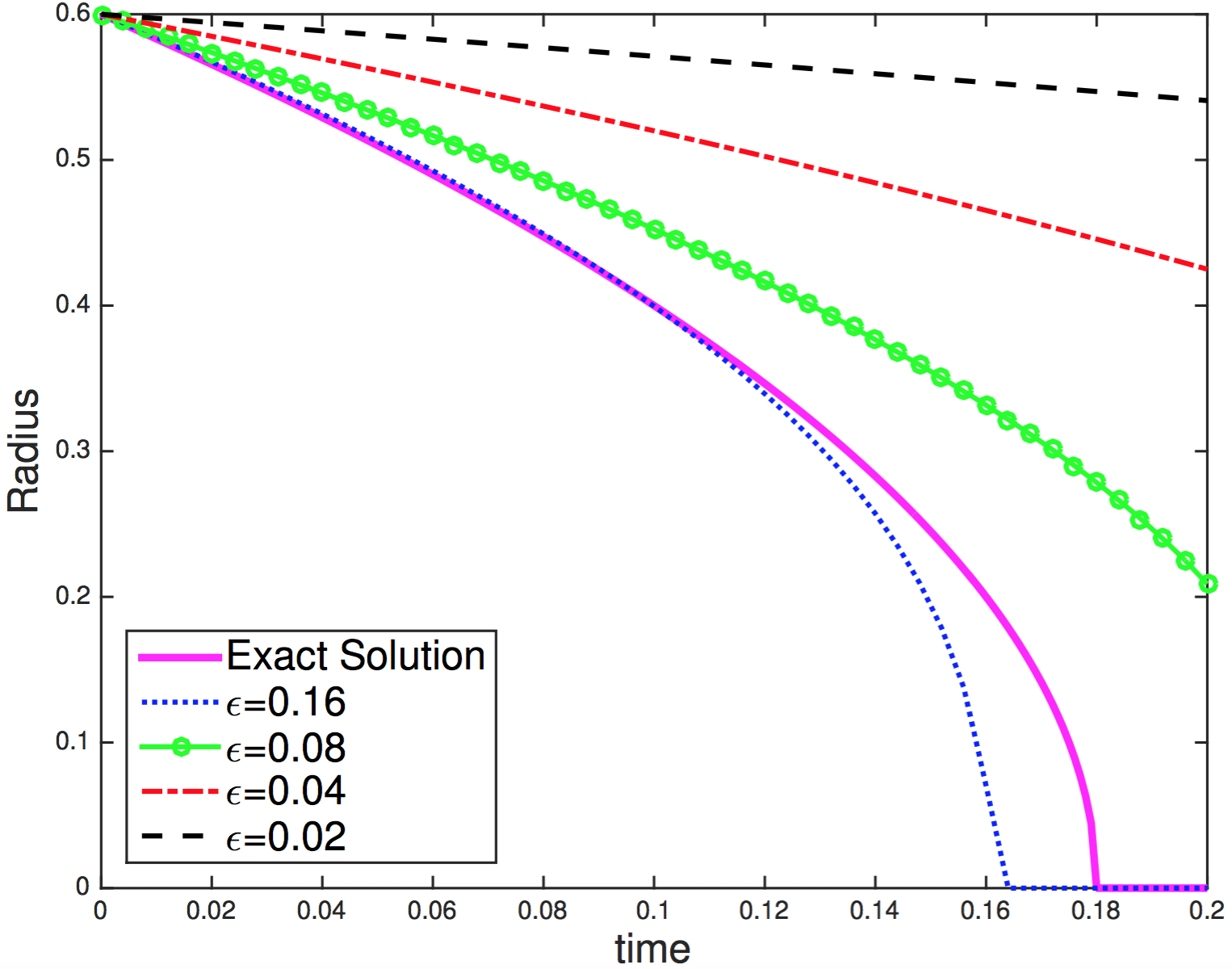}
  \label{fig:AC-diff-eps-CCS}
}
\caption{Allen-Cahn equation: FIS and CSS with $k_n=0.002,\ h=1/256$ and
  different $\epsilon$'s.} 
  \label{fig:AC-diff-eps}
\end{figure}

\paragraph{Test 2} \label{test2}
In this simulation, we minimize the discrete energy
\eqref{eq:FIS-AC-energy} for the Allen-Cahn equation at each time
step. The computational domain is $\Omega=(-1,1)^2$, and
parameter is $\epsilon =5\times 10^{-3}$. The initial value, shown
in Figure \ref{fig:4.1}, is chosen as
\begin{equation}
u_0(x,y) = \tanh\left(\frac{\sqrt{x^2+y^2}-0.6}{\sqrt{2}\epsilon}
\right).
\end{equation}
When $t$ increases, we expect the radius of the hole to decrease, as
shown in Figure \ref{fig:4.2}.

\begin{figure}[!htbp]
\centering 
\subfloat[Initial value of $u$]{
  \includegraphics[scale=0.30]{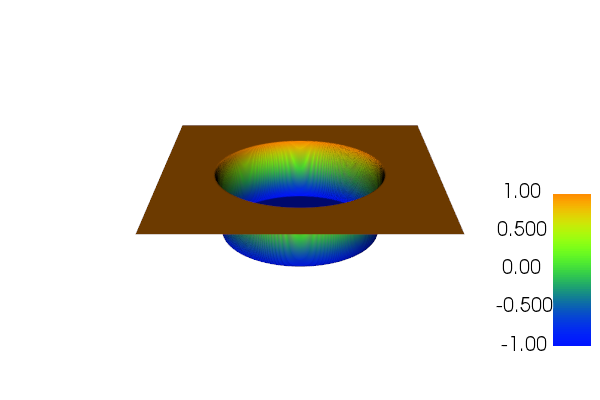}
  \label{fig:4.1}
}
\subfloat[Value of $u$ at $t=0.14$]{
  \includegraphics[scale=0.30]{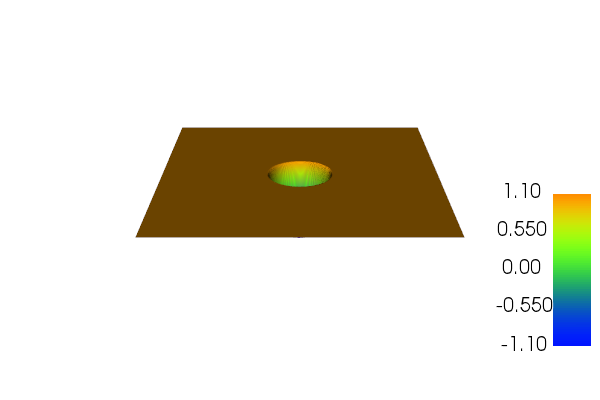}
  \label{fig:4.2}
}
\caption{The Allen-Cahn equation with smooth initial value: Values of
$u$ at different $t$'s.}
\end{figure}

Our goal is to test if the solution from the energy minimization scheme
approximates the physical solution even when the discrete energy is
non-convex. Recall that when $k_n\leq \epsilon^2$, the discrete energy
is convex.

We first test the dependency on the initial guess for the L-BFGS
minimization algorithm (cf. \cite{nocedal1980updating,
byrd1995limited}).  Here we choose $k_{1} = 10^{-3}$, which leads
to the non-convex discrete energy \eqref{eq:energy-AC}. Figure
\ref{fig:4.1b} shows the global minimizer by using $u_0(x,y)$ as the
initial guess for the L-BFGS, which is quite similar to the
solution obtained with $k_n=10^{-5}$, see Figure \ref{fig:4.3}; Figure
\ref{fig:4.2b} shows a local minimizer by using the initial guess for
L-BFGS as $1-u_0(x,y)$. We observe that when the initial guess is the
solution at previous time step, the local minimizer has lowest
discrete energy, and that the solution with the lowest discrete energy
is the best approximation to the solution obtained in the convex case.

\begin{figure}[!htbp]
\centering 
\subfloat[Global minimum with $E_{1}^{\rm AC} = 718.9$]{
  \includegraphics[scale=0.4]{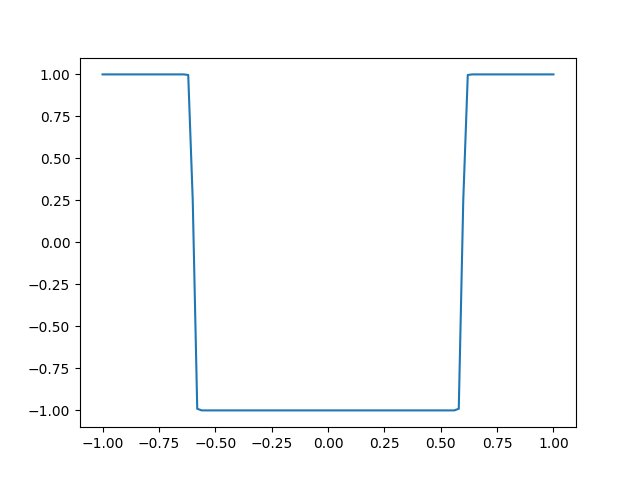}
  \label{fig:4.1b}
}
\subfloat[Local minimum with $E_{1}^{\rm AC} = 2206.4$]{
  \includegraphics[scale=0.4]{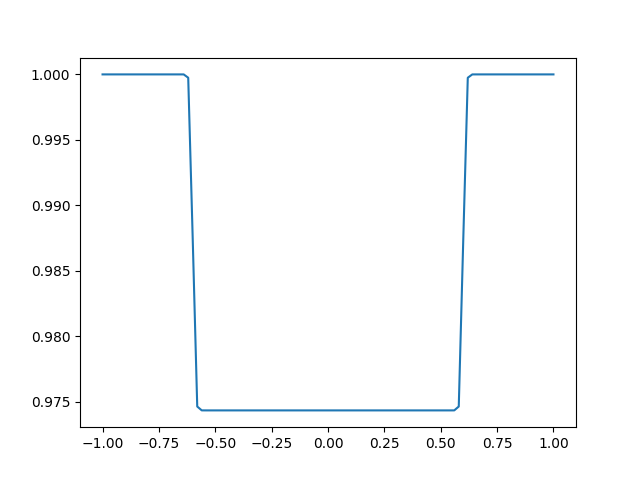}
  \label{fig:4.2b}
}
\caption{The Allen-Cahn equation with smooth initial value: Minimizers
at $t=10^{-3}$ for different initial guesses in the L-BFGS
algorithm, $k_{1} = 10^{-3}$.}
\end{figure}

We compare the solution for $k_n=10^{-5}$, for the case in
which the discrete energy is convex, with the solutions for
$k_n=10^{-4}$ and $k_n=10^{-3}$, for the cases in which the discrete
energies are non-convex. For the L-BFGS
algorithm, the initial guess is set to be the solution at previous
time step.  Figure \ref{fig:4.3} displays the cross-sectional
solutions at $y=0$ at different $t$'s.  We observe that energy
minimization version of fully implicit schemes performs all well with
different time step sizes. 

\begin{figure}[!htbp]
\centering 
\subfloat[$t=0$]{
  \includegraphics[scale=0.4]{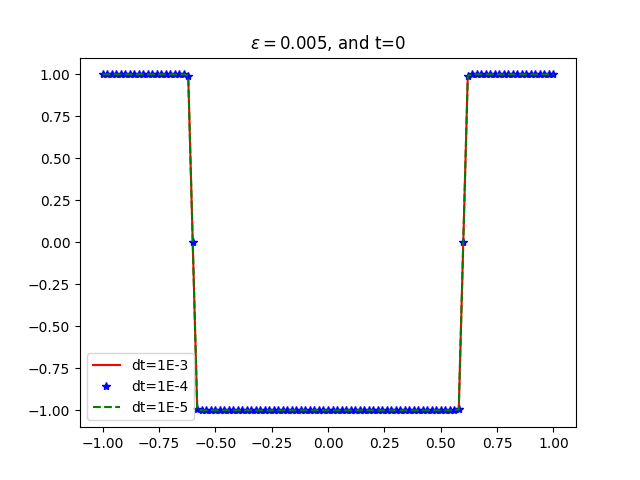}
  \label{fig:4.3-1}
}%
\subfloat[$t=0.05$]{
 \includegraphics[scale=0.4]{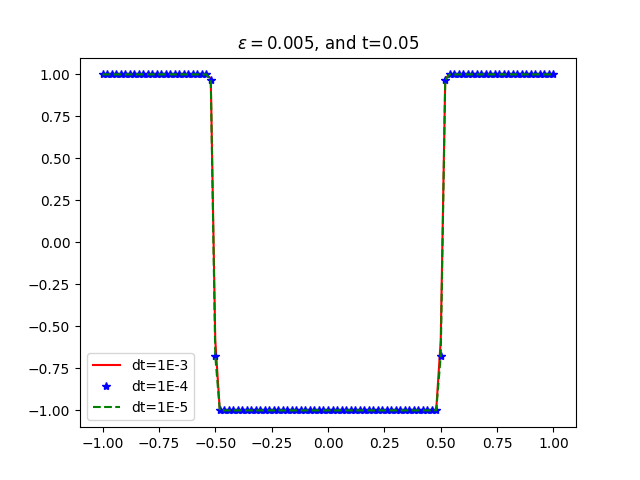}
  \label{fig:4.3-2}
} \\
\subfloat[$t=0.1$]{
 \includegraphics[scale=0.4]{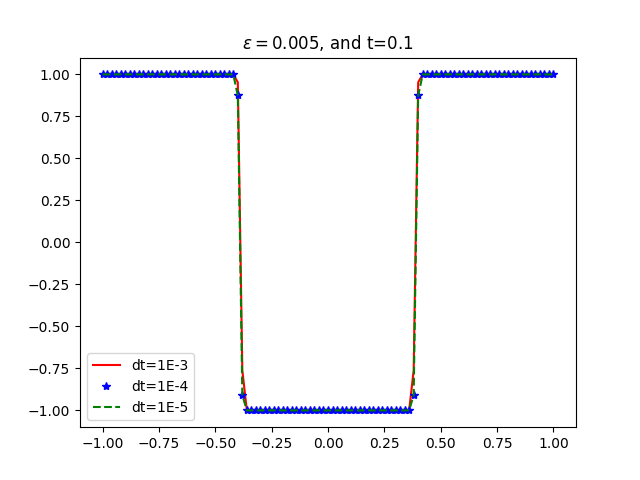}
  \label{fig:4.3-3}
}% 
\subfloat[$t=0.14$]{
 \includegraphics[scale=0.4]{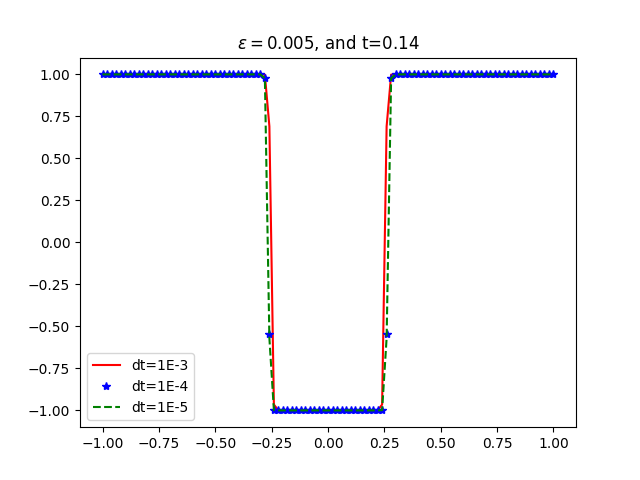}
  \label{fig:4.3-4}
}% 
\caption{The Allen-Cahn equation with smooth initial value: Plot of
  the cross-sectional solutions $u_h(x,0)$ at different $t$'s.}
\label{fig:4.3}
\end{figure}

Since the initial guess for the L-BFGS algorithm is random, we
conclude that the L-BFGS algorithm does not depend on the initial
guess when the solution is smooth enough.  We also compare the evolutions of
physical energies $J_\epsilon^{\rm AC}$ for the three cases in Figure
\ref{fig:4.4}, which shows that the energy minimization version of
fully implicit scheme is energy-stable. 

\begin{figure}[!htbp]
\begin{center}
 \includegraphics[scale=0.6]{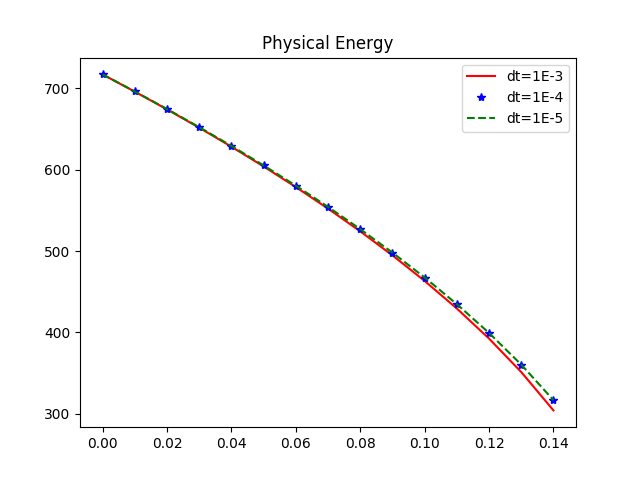}
\end{center}
\caption{The Allen-Cahn equation with smooth initial value: Evolutions
of the physical energies.}
\label{fig:4.4}
\end{figure}

\paragraph{Test 3} \label{test3}
In this set of simulations, we minimize the discrete energy
\eqref{eq:FIS-AC-energy} for the Allen-Cahn equation with
$\epsilon = 5\times 10^{-3}$ at each time step. The computational
domain is chosen as $\Omega=(-1,1)^2$, while the initial
value $u_0(x,y)$ is randomly chosen. 

In order to smooth the initial value, we first compute the solution
from $t = 0$ to $t = 2\times 10^{-3}$ with $k = 10^{-5}$, namely 
$$ 
k_n = 10^{-5} \qquad \text{for }n = 1, 2, \cdots, 200.
$$ 
Then, we switch for different time step sizes with the energy
minimization version of fully implicit scheme. This is needed only when $k_n \geq 10^{-3}$.

After the smoothing the random initial value, we first test the
dependency on the initial guess for the L-BFGS minimization algorithm. 
Here we choose $k_{201} = 10^{-3}$, which leads to the non-convex
discrete energy \eqref{eq:FIS-AC-energy}. The reference solution is
obtained by evolving the Allen-Cahn equation with $k_n = 10^{-5}$
(convex case). Figure \ref{fig:global-local-minimizer} shows the
different local minimizers from different initial guesses. We observe
that: (1) When the initial guess is the solution at previous time
step, the local minimizer has lowest discrete energy; (2) The solution
with the lowest discrete energy is the best approximation to the
reference solution; (3) When the initial guesses are random chosen, we
obtain several different local minimizers. This implies that the
result obtained from L-BFGS does depend on the initial guess when the
solution is not smooth enough. Therefore, we will (and recommend to)
choose the solution at previous time step as the initial guess for the
L-BFGS algorithm.

\begin{figure}[!htbp]
\centering 
\subfloat[Reference solution]{
  \includegraphics[scale=0.25]{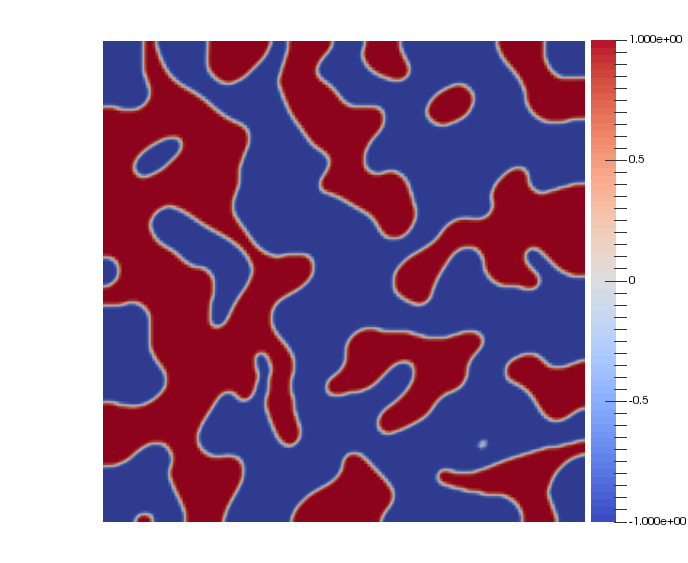} 
  \label{fig:global-local-ref}
}%
\subfloat[$u_h^{200}$ as initial guess, $E_{201}^{\rm AC} = 3848.7$]{
  \includegraphics[scale=0.25]{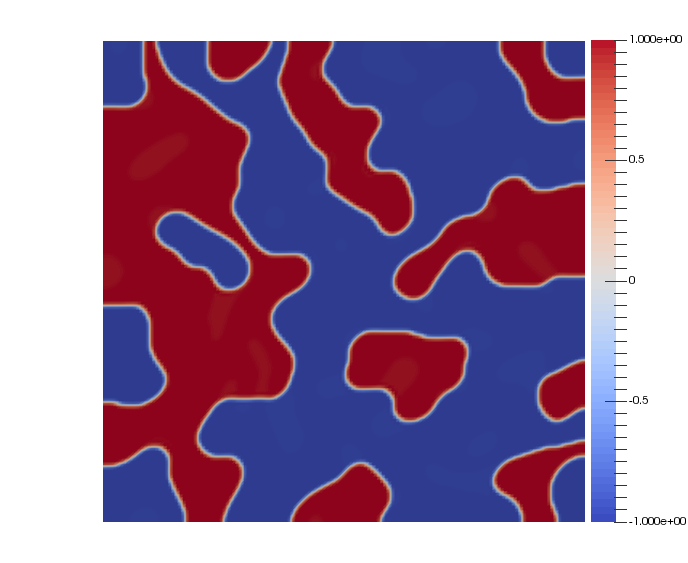} 
  \label{fig:global-local-global}
} \\
\subfloat[Random initial guess, $E_{201}^{\rm AC} = 4238.5$]{
  \includegraphics[scale=0.25]{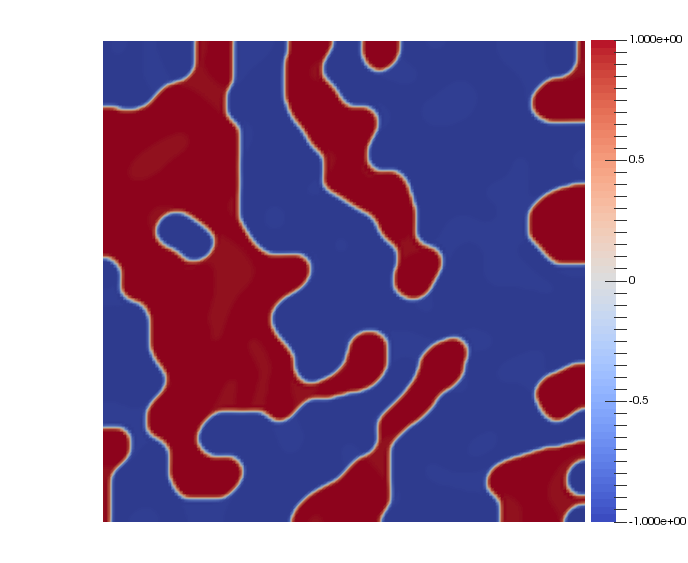} 
  \label{fig:global-local-local1}
}% 
\subfloat[Random initial guess, $E_{201}^{\rm AC} = 4341.7$]{
  \includegraphics[scale=0.25]{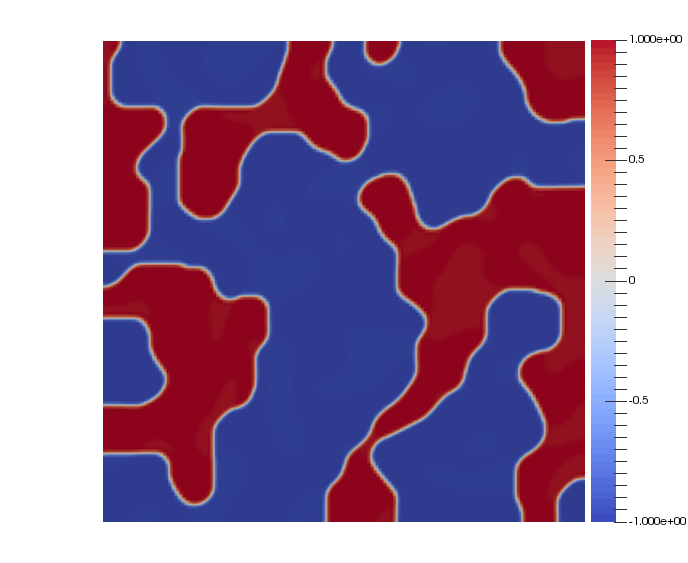}
  \label{fig:global-local-local2}
}% 
\caption{The Allen-Cahn equation with random initial value: Minimizers
at $t=3\times 10^{-3}$ for different initial guesses in the L-BFGS
algorithm, $k_{201} = 10^{-3}$.}
\label{fig:global-local-minimizer}
\end{figure}

Next, we evolve the Allen-Cahn equation with different time step sizes 
after $t = 2\times 10^{-3}$ to see the two phases regroup. Three
different computations with $k_n = 10^{-5}$ (convex case), $k_n =
10^{-4}$ and $k_n = 10^{-3}$ (non-convex cases) are considered. In
Figure \ref{fig:4.5} shows the random initial value and the
evolutions of the numerical solutions at different $t$'s. It can be
observed that the solutions in all these cases behave similarly.    
In addition, for the given random initial condition, the evolution of
physical solution and physical energy seem a little
bit faster than the others when choosing $k_n = 10^{-3}$, as shown in
Figure \ref{fig:random-energy}. This is most likely because of the
time discretization error for the large time step size. Furthermore,    
the evolutions of the physical energies show the energy-stability of
the energy minimization version of the fully implicit scheme, which is
in agreement with the Theorem \ref{thm:convexity-FIS-AC}.

\begin{figure}[!htbp]
\centering 
\subfloat[$t=0$, $k_n = 10^{-5}$]{
  \includegraphics[scale=0.17]{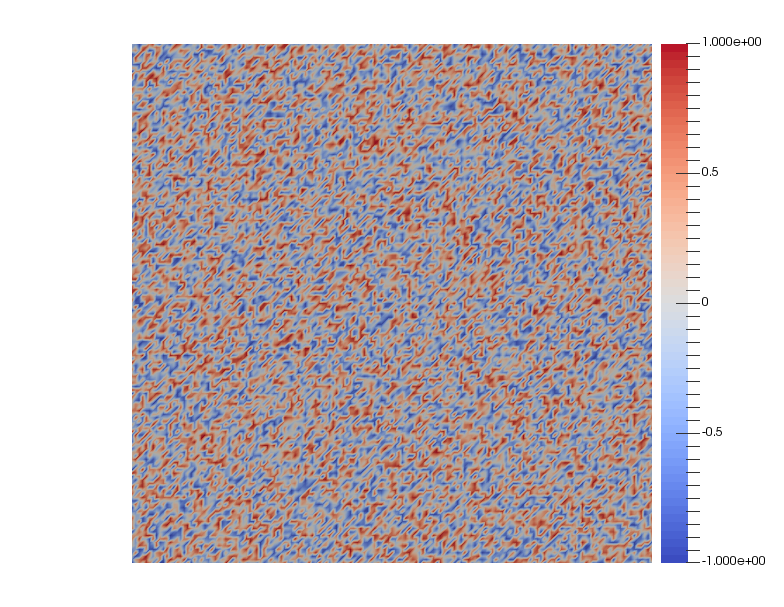}
}%
\subfloat[$t=0$, $k_n = 10^{-4}$]{
  \includegraphics[scale=0.17]{1e-5_t=0.png}
}%
\subfloat[$t=0$, $k_n = 10^{-3}$]{
  \includegraphics[scale=0.17]{1e-5_t=0.png}
} \\
\subfloat[$t=0.02$, $k_n = 10^{-5}$]{
  \includegraphics[scale=0.17]{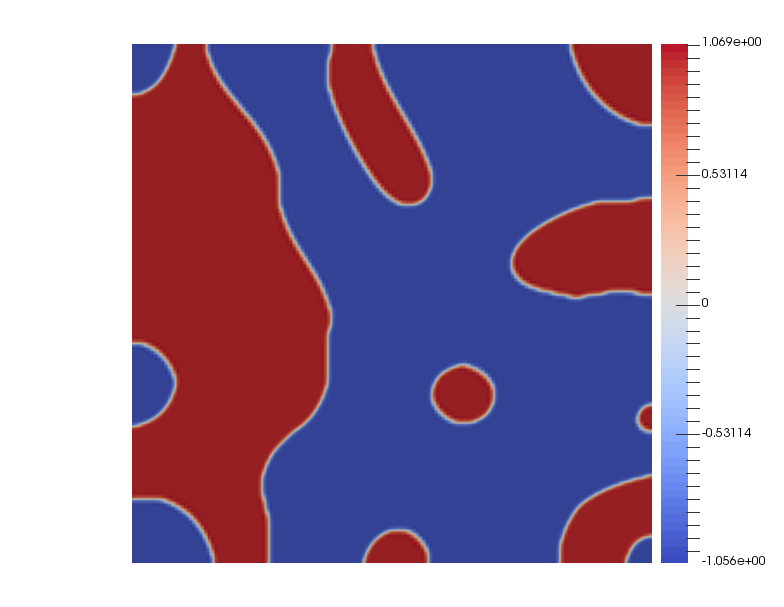} 
}%
\subfloat[$t=0.02$, $k_n = 10^{-4}$]{
  \includegraphics[scale=0.17]{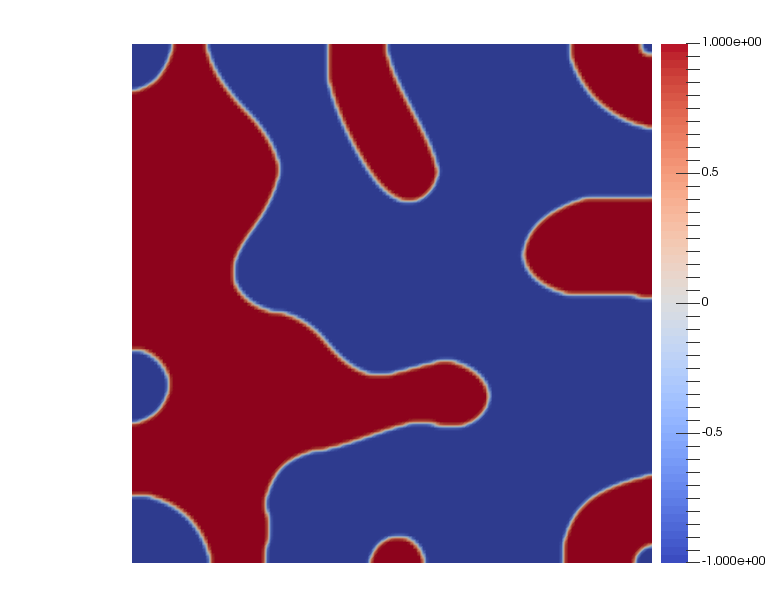}
}%
\subfloat[$t=0.02$, $k_n = 10^{-3}$]{
  \includegraphics[scale=0.17]{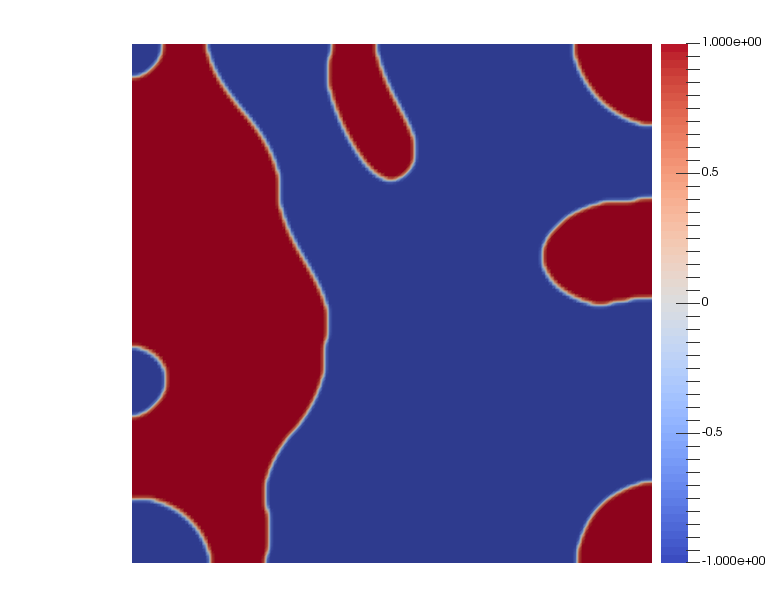}
} \\
\subfloat[$t=0.05$, $k_n = 10^{-5}$]{
  \includegraphics[scale=0.17]{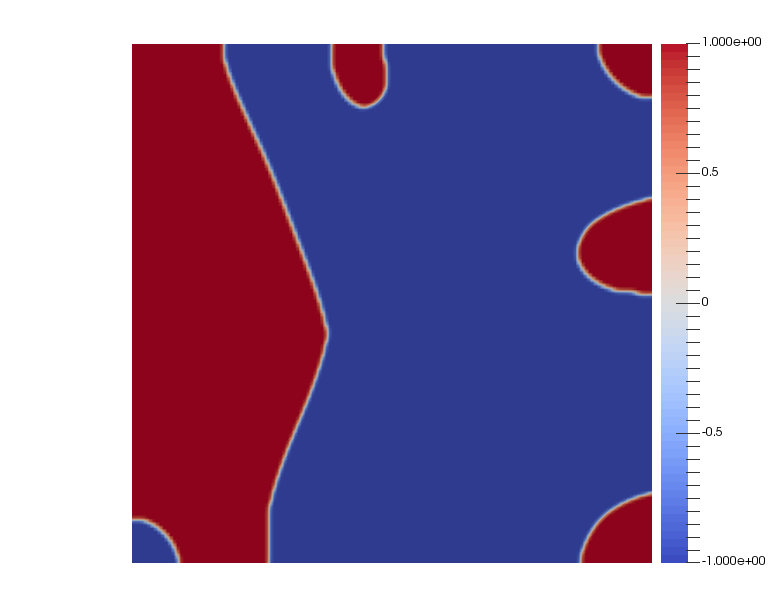}
}%
\subfloat[$t=0.05$, $k_n = 10^{-4}$]{
  \includegraphics[scale=0.17]{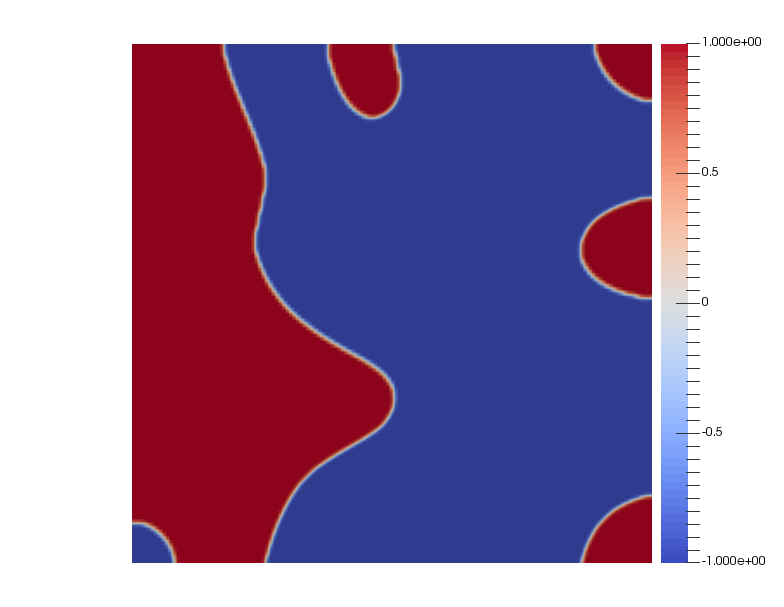}
}%
\subfloat[$t=0.05$, $k_n = 10^{-3}$]{
  \includegraphics[scale=0.17]{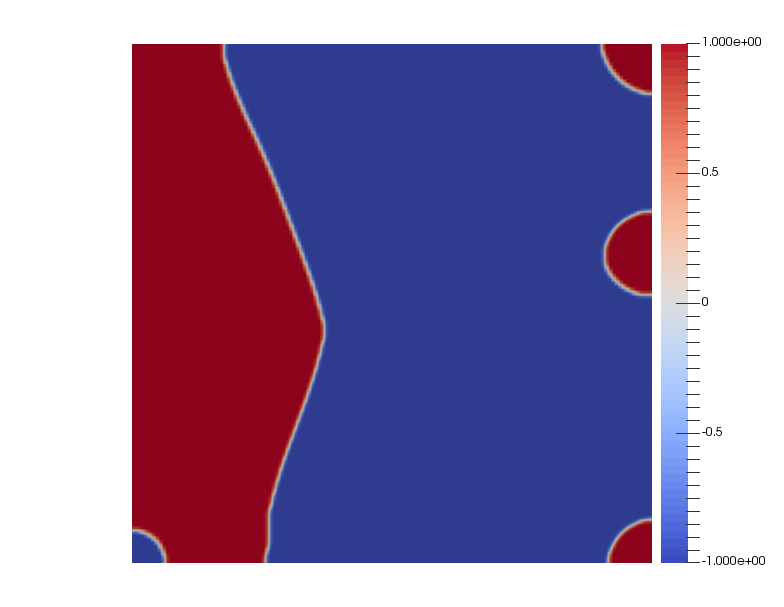}
}\\
\subfloat[$t=0.14$, $k_n = 10^{-5}$]{
  \includegraphics[scale=0.17]{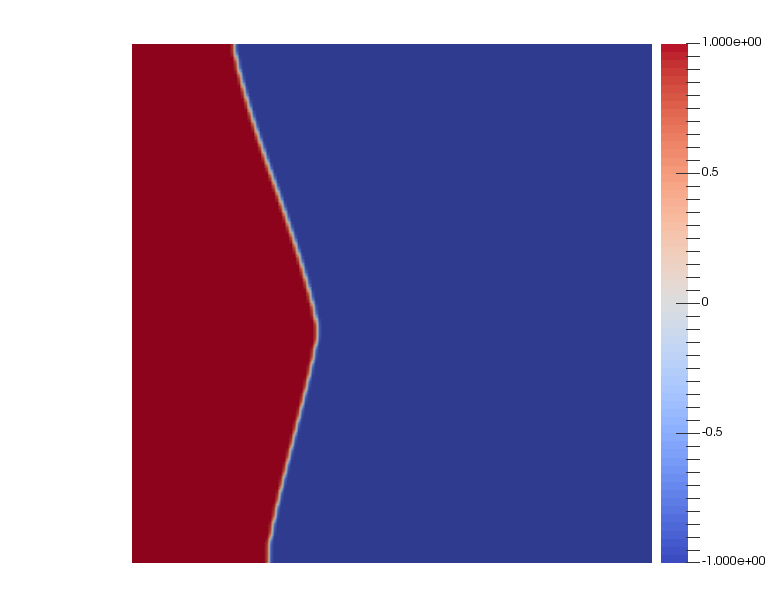}
}%
\subfloat[$t=0.14$, $k_n = 10^{-4}$]{
  \includegraphics[scale=0.17]{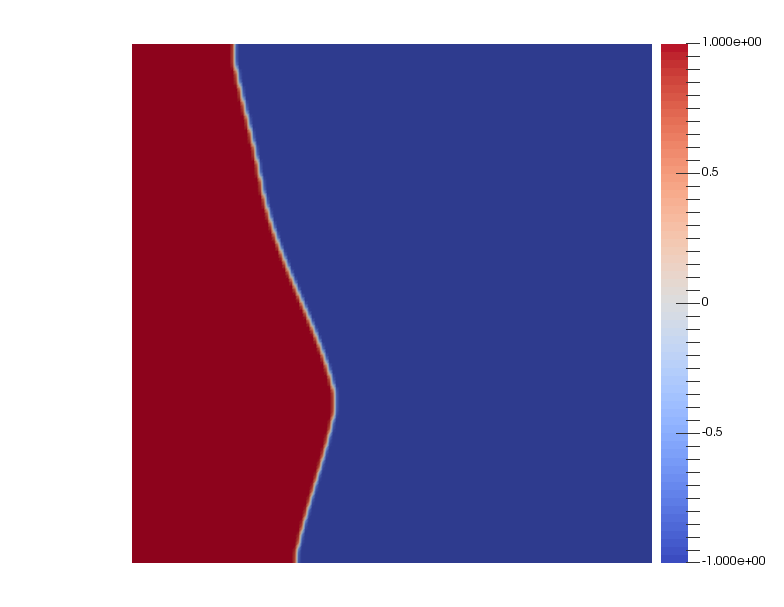}
}%
\subfloat[$t=0.14$, $k_n = 10^{-3}$]{
  \includegraphics[scale=0.17]{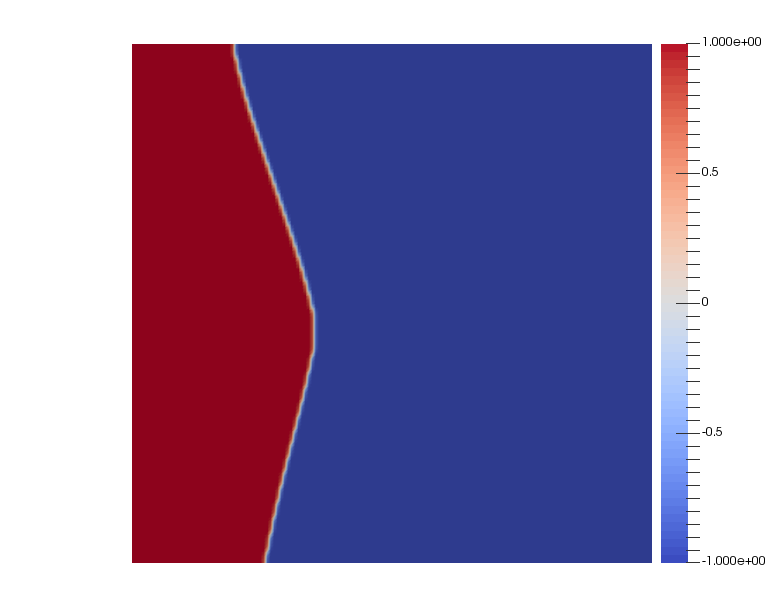}
}% \\

\caption{The Allen-Cahn with random initial value: Plot of the
  solutions at different $t$'s.}
\label{fig:4.5}
\end{figure}

\begin{figure}[!htbp]
\begin{center}
 \includegraphics[scale=0.6]{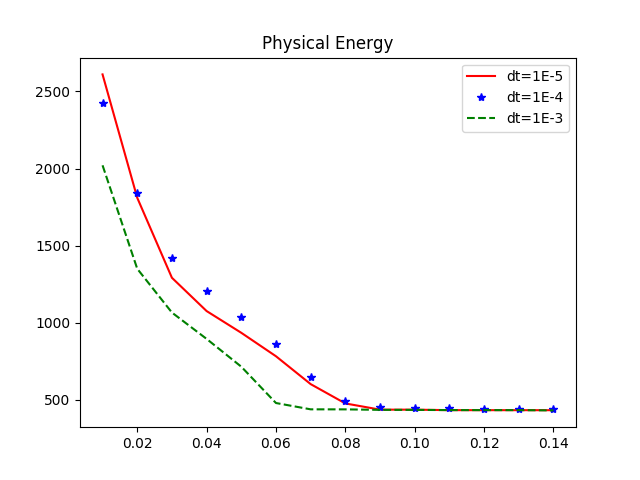}
\end{center}
\caption{The Allen-Cahn equation with random initial guess: Evolutions
  of physical energies.}
\label{fig:random-energy}
\end{figure}

\subsubsection{A convex splitting scheme for the Cahn-Hilliard model}
Similar to the Allen-Cahn model, a convex splitting scheme can also be
obtained for Cahn-Hilliard model as follows: Find $u_h^{n}\in V_h$ for
$n=1,2,\cdots, N$, such that
\begin{equation} \label{eq:CSS-CH}
\begin{aligned}
(\frac{u_h^{n}-u_h^{n-1}}{k_n},\eta_h)+(\nabla w_h^{n},\nabla \eta_h) &=0
\qquad \forall \,\eta_h\in V_h,\\
\epsilon (\nabla u_h^{n},\nabla v_h) + \frac{1}{\epsilon}((u_h^n)^3 -
u_h^{n-1},v_h)-(w_h^{n},v_h) &=0 \qquad \forall\,
  v_h\in V_h. 
\end{aligned}
\end{equation}

\begin{theorem}
The Discretization of the Cahn-Hilliard equation using the convex
splitting scheme is equivalent to the discretization of the following
equations using the fully implicit scheme: 
\begin{equation} \label{eq:perturbed-CH}
\begin{aligned}
u_t - \Delta w&= 0, \\
w + \epsilon \Delta u -
\frac{1}{\epsilon}f(u)-\frac{k_n}{\epsilon}u_t
&=0.
\end{aligned}
\end{equation}
\end{theorem}

We note that \eqref{eq:perturbed-CH} can be equivalently written as
follows:
\begin{equation} \label{ConvexCH}
(1 - \frac{k_n}{\epsilon}\Delta)u_t + \Delta (\epsilon \Delta u -
\frac{1}{\epsilon}f(u)) = 0.
\end{equation}
It is known that \cite{caginalp1998convergence} when
$k_n=\mathcal{O}(\epsilon^3)$, the solution of \eqref{ConvexCH}
converges to the Hele-Shaw flow, which is also the limiting dynamics
for the Cahn-Hilliard equation \eqref{eq:CH}.   In other situations,
for example, when $k_n = \mathcal{O}(\epsilon^2)$, their limiting
dynamics may be different.

\paragraph{Test 4} \label{test4} 
In this test, the computational domain is $(0,1)^2$, and the
following initial condition for the Cahn-Hilliard equation is chosen
as
\begin{align}\label{eq:initial-smooth2} u(x,t) =
\mathrm{tanh}\bigg(\frac{\sqrt{x^2+y^2}-0.17}{\sqrt{2}\epsilon}\bigg),
\qquad \epsilon = 0.02.
\end{align} 
Again, the Figure \ref{fig:CH-delay} is the snapshot showing the
lagging phenomenon at different time points.

\begin{figure}[!htbp]
\centering
\includegraphics[width=0.35\textwidth]{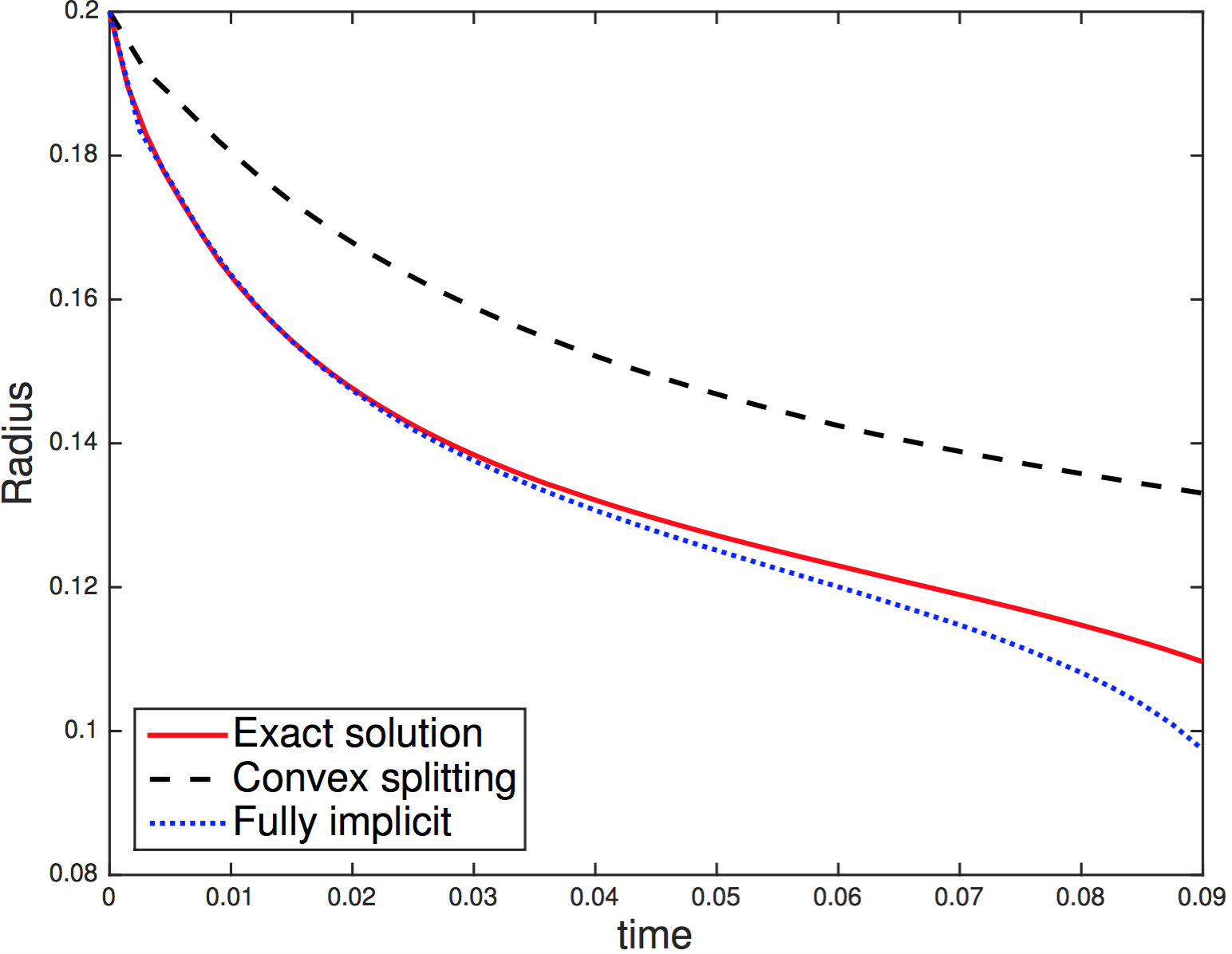}
\caption{Cahn-Hilliard equations: FIS and CSS. Here,
  $\epsilon=0.02, k_n=5\times10^{-4}$ and $h=0.015$.}\label{fig:CH-delay}
\end{figure}

\subsection{Some other first-order partially implicit schemes}
In this section, we briefly discuss several other first-order
partially implicit schemes for the Allen-Cahn model. 

{\it Semi-implicit scheme}: Seeking $u_h^n\in V_h$ for $n=1,2,\cdots$, such that
\begin{align}\label{semi-implicit}
(\frac{u_h^n - u_h^{n-1}}{k_n}, v_h) + (\nabla u_h^n, \nabla v_h) +
\frac{1}{\epsilon^2}(f(u_h^{n-1}),v_h) = 0 \qquad\forall v_h\in V_h.
\end{align}

%% Stabilized semi-implicit
{\it Stabilized semi-implicit scheme}: Seeking $u_h^{n}\in V_h$ for
$n=1,2,\cdots$, such that
\begin{align}\label{stabilized-semi-implicit}
(\frac{1}{k_n}+\frac{S}{\epsilon^2})(u_h^n-u_h^{n-1},v_h)+(\nabla
u_h^{n},\nabla v_h)+\frac{1}{\epsilon^2}(f(u_h^{n-1}),v_h)=0
\qquad\forall v_h\in V_h,
\end{align}
where $S>0$ (set as $S=1$ in the Test 5) is a stabilized constant.

\begin{theorem} \label{thm:time-scaling}
The scheme \eqref{semi-implicit} and
\eqref{stabilized-semi-implicit} can be recast as 
\begin{equation} \label{unified-AC}
\left( \frac{1+\gamma_n}{k_n}(u_h^n - u_h^{n-1}), v_h\right) + (\nabla
u_h^n, v_h) + \frac{1}{\epsilon^2} (f(u_h^n), v_h) = 0 \qquad
\forall v_h\in V_h.
\end{equation}
For semi-implicit scheme \eqref{semi-implicit}, 
$$ 
\gamma_n =  \frac{k_n}{\epsilon^2}[1 - (u_h^n)^2 - u_h^n u_h^{n-1} -
  (u_h^{n-1})^2],
$$
and for stabilized semi-implicit scheme
\eqref{stabilized-semi-implicit},
$$ 
\gamma_n = \frac{k_n}{\epsilon^2}[1+S - (u_h^n)^2 - u_h^n u_h^{n-1} -
  (u_h^{n-1})^2].
$$
\end{theorem}
\begin{proof}
 For semi-implicit and stabilized semi-implicit schemes, the parameter $\delta_n$ can
 be derived from $f(u_h^{n-1}) = f(u_h^n) + [1 - (u_h^n)^2 - u_h^n
 u_h^{n-1} - (u_h^{n-1})^2](u_h^n - u_h^{n-1})$.
\end{proof}

Depending on the size and sign of $\gamma_n$, the above theorem will
offer some insight to the behavior of the two semi-implicit schemes in
comparison with the fully implicit scheme \eqref{eq:FIS-AC}. 

\paragraph{Test 5} \label{test5} 
In this test, the same domain and initial conditions are chosen as in
Test 1. On the left graph of Figure \ref{fig:AC-delay}, the same
$\epsilon$, $h$ and $k$ are chosen to draw the graphs using different
numerical schemes comparing with the exact solution (which is obtained
by highly refined meshes and extremely small time step size). We
observe that only the FIS performs well. The right graph shows the
delayed convergence" of the CSS.

\begin{figure}[!htbp]
\centering 
\captionsetup{justification=centering}
\subfloat[Different schemes: $\epsilon=0.02,\ k_n=0.0005$ and $h=0.015$]{
  \includegraphics[width=0.35\textwidth]{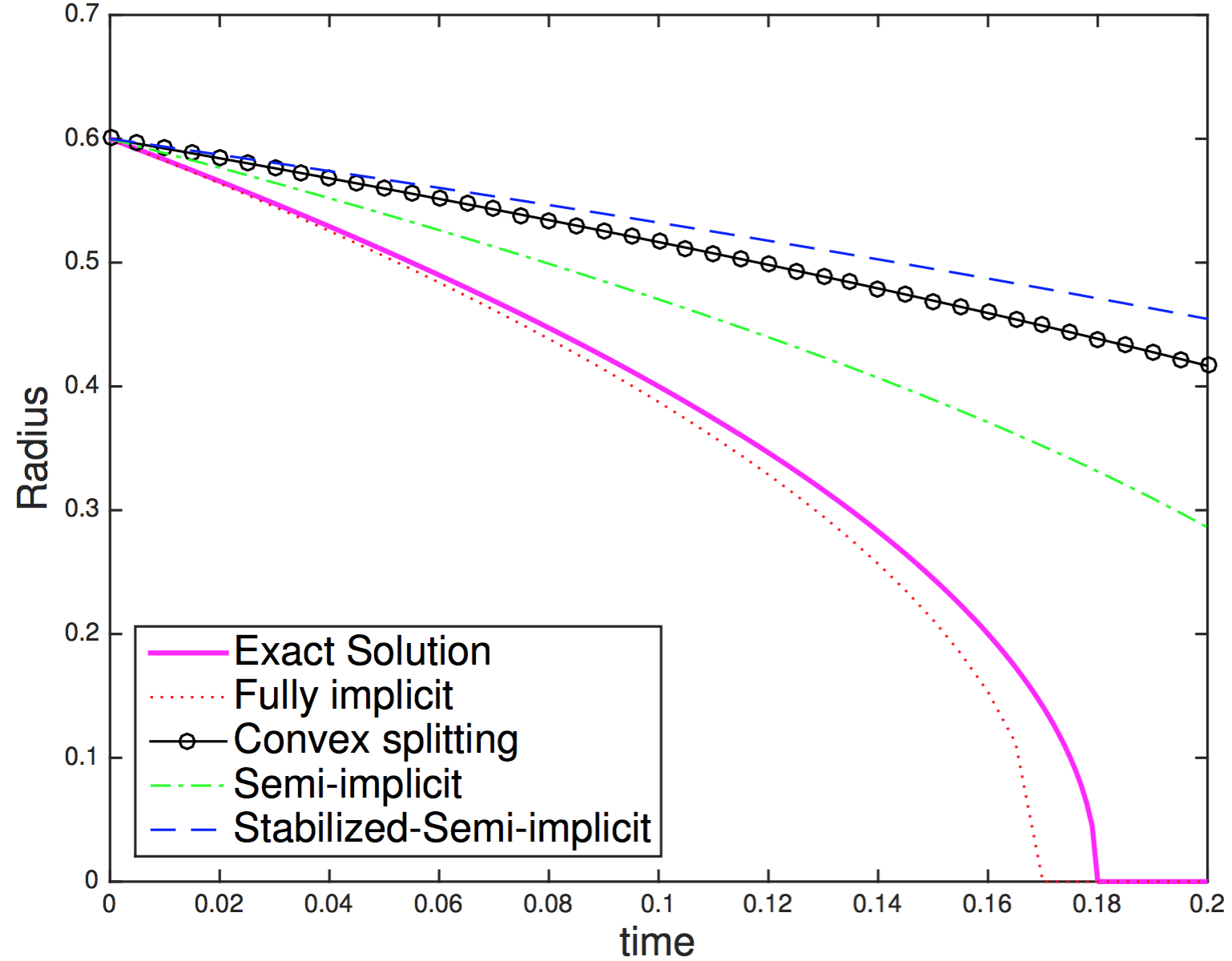}
}%
\subfloat[CSS: $\epsilon=0.02,\ h=0.015$ and $k_n$'s]{
  \includegraphics[width=0.35\textwidth]{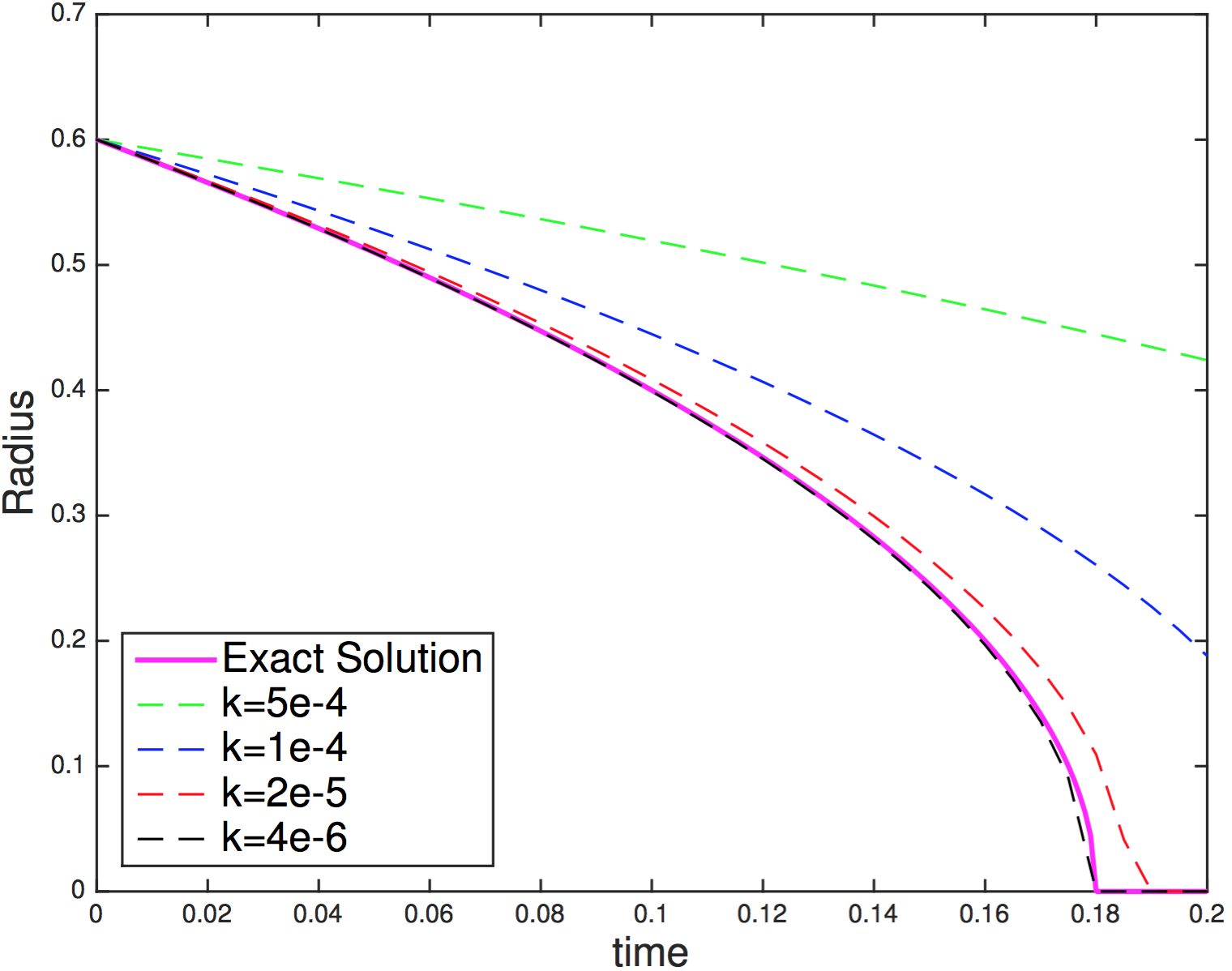}
}
\caption{Allen-Cahn equation: Radius change with time using different
numerical methods.} \label{fig:AC-delay}
\end{figure}

\subsection{Convex splitting schemes interpreted as artificial convexity
schemes} \label{subsec:convexity}
In this section, we give a slightly different perspective on convex
splitting schemes. We consider the following modified Allen-Cahn
model:
\begin{equation} \label{ConvexAC-delta}
\bigg(1+\frac{\delta_n}{\epsilon^2}\bigg)u_t-\Delta u+\frac{1}{\epsilon^2}f(u)=0,  
\end{equation} 
and the following modified Cahn-Hilliard model:
\begin{equation}\label{ConvexCH-delta}
\begin{aligned}
\bigg(1-\frac{\delta_n}{\epsilon}\Delta\bigg)u_t-\Delta w &=0  \qquad
\mbox{in } \Omega_T,\\
-\epsilon\Delta u +\frac{1}{\epsilon}f(u) &=w \qquad \mbox{in }
\Omega_T.
\end{aligned}
\end{equation}

%\begin{alignat}{2}
%u_t - \Delta w&= 0  &&\qquad \mbox{in }
%\Omega_T,\\
%w + \epsilon \Delta u - \frac{1}{\epsilon}f(u)-\frac{\delta}{\epsilon}u_t
%&=0 &&\qquad \mbox{in }
%\Omega_T.
%\end{alignat}

\begin{theorem}\label{thm:convexification}
When $k _n\leq \epsilon^2+\delta_n$, the standard fully implicit scheme
for \eqref{ConvexAC-delta} is equivalent to the convex minimization
problem:
\begin{equation} \label{eq:MAC-FIS}
u_h^n = \underset{u_h\in V_h}{\mathrm{argmin}}
\left\{
J_\epsilon^{\rm AC}(u_h)+(\frac{1}{2k_n}+\frac{\delta_n}{2k_n\epsilon^2})\int_{\Omega}(u_h
- u_h^{n-1})^2dx \right\}.
\end{equation}
When $k_n\le(\epsilon^{3\slash2}+\sqrt{\epsilon^3+\delta_n})^2$, the
standard fully implicit scheme for \eqref{ConvexCH-delta} is
equivalent to the convex minimization problem: 
\begin{equation} \label{eq:MCH-FIS}
u_h^n = u_h^{n-1} + \theta_h^n, \quad \theta_h^n = \underset{\theta_h\in
\mathring{V}_h}{\mathrm{argmin}}
\left\{
J_\epsilon^{\rm CH}(u_h^{n-1}+\theta_h) +
\frac{1}{2k_n}\|\nabla\Delta_h^{-1}\theta_h\|_{L^2(\Omega)}^2 +
\frac{\delta_n}{2k_n\epsilon}\|\theta_h\|_{L^2(\Omega)}^2
\right\}.
\end{equation}
\end{theorem}

\begin{proof}
The proofs of \eqref{eq:MAC-FIS} and \eqref{eq:MCH-FIS} are similar to
Theorem \ref{thm:convexity-FIS-AC} and \ref{thm:convexity-FIS-CH},
respectively. 
\end{proof}

In view of Theorem \ref{thm:convexification}, the modified model
\eqref{ConvexAC-delta} may be viewed as a convexified model of the original
Allen-Cahn model \eqref{eq:AC}; the added term
$\frac{\delta_n}{\epsilon^2}u_t$ introduces a new time scale of the
model and on the discrete level it plays the role of an artificial
convexification.  Similarly, the modified model \eqref{ConvexCH-delta}
may be viewed as a convexified model of the original Cahn-Hilliard
model \eqref{eq:CH}.  We note that the CSS for the original Allen-Cahn
or Cahn-Hilliard model is the FIS for the convexified model with
$\delta_n =k_n$. 

With such an interpretation, the convex splitting scheme may be more
appropriately viewed as an artificial convexity scheme.  This is in
some way similar to the artificial viscosity scheme for hyperbolic
equations or convection dominated convection-diffusion problems. The
physical implication of the convexified model
\eqref{ConvexAC-delta} is a new time-scale: $t'=
(1+\frac{\delta_n}{\epsilon^2})t$, which leads to a time-delay in
comparison to the original model. The implication of the modified
model \eqref{ConvexCH-delta} seems to be similar but less obvious.

\section{A modified FIS satisfying a discrete maximum principle}
\label{sec:modified-FIS}
%%In section \ref{sec:CSS-FIS}, we demonstrate that the convex splitting
%%scheme is nothing but the fully implicit scheme with a reduced time
%%step size for the Allen-Cahn equation. 
In this section, we will modify
the fully implicit scheme (or the corresponding convex splitting
scheme) to preserve the maximum principle on discrete level. We will
then further show that this modified scheme can be uniformly
preconditioned by a Poisson-like operator. We refer to
\cite{nochetto1997convergence, shen2014maximum} for other maximum
principle preserving schemes for the Allen-Cahn equation.

\subsection{A modified scheme}
Our modified  FIS is motivated by the maximum principle of Allen-Cahn
on continuous level stated in the following theorem (see
\cite{gilbarg2015elliptic, feng2007analysis} for the idea, and  
Proposition 2.2.1 in \cite{li2015numerical} for the details).
\begin{theorem} \label{thm:maximum-AC-continuous}
If $u$ is a weak solution of the Allen-Cahn equation \eqref{eq:AC} and
$\|u_0\|_{L^{\infty}(\Omega)} \leq 1$, then
$\|u(x,t)\|_{L^{\infty}(\Omega)} \leq 1$.
\end{theorem}
Unfortunately, the above maximum principle can not be proved for a
standard FIS.  In this section, we will modify the standard FIS scheme
so that a maximum principle preserving scheme analogous to
Theorem~\ref{thm:maximum-AC-continuous} can also be rigorously proved.

We consider the $P_1$-Lagrangian finite element space in this section,
$$ 
V_h = \bigl\{v_h \in C(\bar{\Omega}): v_h|_{K} \in P_1(K)\bigr\}. 
$$ 
The nodal basis function of $V_h$ related to the vertex $a_i$ is
denoted as ${\varphi_i}$. We then define the nodal value interpolation
$I_h: C(\bar{\Omega}) \mapsto V_h$ as 
\begin{equation} \label{interpolation}
I_h v := \sum_{a_i \in \mathcal{N}_h} v(a_i)\varphi_i = \sum_{a_i \in
\mathcal{N}_h} v_i \varphi_i.
\end{equation} 

Following \cite{xu1999monotone}, for given $K \in \mathcal{T}_h$ , we
introduce the following notation: $a_i (1\leq i \leq n+1)$ denote the
vertices of $K$, $E=E_{ij}$ the edge connecting two vertices $a_i$ and
$a_j$, $F_i$ the $(n-1)$-dimensional simplex opposite to the vertex
$a_i$, $\theta_{ij}^K$ or $\theta_E^K$ the angle between the faces
$F_i$ and $F_j$, $\kappa_E^K=F_i \cap F_j$ , the $(n-2)$-dimensional
simplex opposite to the edge $E=E_{ij}$.

We first consider the simplest and important case of the Poisson
equation with Neumann boundary condition. Then, for any $u_h, v_h \in
V_h$, we have (see \cite{xu1999monotone} for details) 
\begin{equation} \label{graph-Laplacian}
(\nabla u_h, \nabla v_h) = \sum_{K \in \mathcal{T}_h} \sum_{E\subset
  K} \omega_E^K \delta_Eu_h \delta_E v_h,
\end{equation}
where $\delta_E \phi = \phi(a_i) - \phi(a_j)$ for any continuous
function $\phi$ on $E = E_{ij}$ and $\omega_E^K =
\frac{1}{n(n-1)}|\kappa_E^K|\cot\theta_E^K$. We will make the
following assumption
\begin{equation} \label{Delaunay}
w_E:= \frac{1}{n(n-1)}\sum_{K\supset E} |\kappa_E^K|\cot \theta_E^K
\geq 0 \qquad
\mbox{ for any edge $E$.}
\end{equation}
We note that, in 2D, the above assumption \eqref{Delaunay} is
equivalent to the Delaunay condition \cite{strang1973analysis} which
requires the sum of any pair of angles facing a common interior edge
to be less than or equal to $\pi$.  For higher dimension a sufficient
condition on ${\mathcal T}_h$ for \eqref{Delaunay} that all the
angles between any two adjacent $(n-1)$-simplicies from ${\mathcal
  T}_h$ are less than or equal to $\frac{\pi}{2}$.
 
With the help of nodal value interpolation, we define a norm
$\|\cdot\|_h$ on $V_h$ as 
\begin{equation} \label{lumping-bilinear} 
\|v_h\|_h^2 := \int_\Omega I_h(v_h^2) ~dx.
\end{equation}
Our {\it modified FIS} is as follows:   Find $u_h^n \in V_h$ for
$n=1,2,\cdots$, such that  
\begin{equation} \label{fully-implicit-lumping}
(\frac{1}{k_n} I_h\big((u_h^{n}-u_h^{n-1})v_h\big), 1) + (\nabla
u_h^{n}, \nabla v_h) +
\frac{1}{\epsilon^2}(I_h\big(f(u_h^{n})v_h\big), 1) = 0
\qquad \forall v_h \in
V_h. 
\end{equation}

\begin{theorem} \label{thm:maximum-AC-lumping}
Assume the triangulation satisfies \eqref{Delaunay}. If
$u_h^n$ is a solution of the modified FIS
\eqref{fully-implicit-lumping} and
$\|u_h^0\|_{L^{\infty}(\Omega)} \leq 1$, then
$\|u_h^n\|_{L^{\infty}(\Omega)} \leq 1$, for all $n\geq 0$. 
\end{theorem}
\begin{proof}
For any function $v \in C(\bar{\Omega})$,  we introduce the following notation:
$$ 
v^+ = \begin{cases}
v & \text{if}~v \geq 0, \\
0 & \text{otherwise},
\end{cases}\qquad \text{and} \qquad  
v^- = \begin{cases}
-v & \text{if}~v \leq 0, \\
0 & \text{otherwise}.
\end{cases}
$$ 
A quick calculation shows that for any $v_i, v_j$,
$$ 
(v_i - v_j)(v_i^+ - v_j^+) - (v_i^+ - v_j^+)^2 = -(v_i^+ - v_j^+)(v_i^- - v_j^-) \geq 0.
$$ 
Therefore, the \eqref{graph-Laplacian} and \eqref{Delaunay} imply 
$$ 
\begin{aligned}
(\nabla v_h, \nabla I_h(v_h^+)) &= \sum_{K \in \mathcal{T}_h} \sum_{E
  \subset K} w_E^K \delta_E v_h \delta_E(I_h(v_h^+)) \\ 
  & \geq \sum_{K \in \mathcal{T}_h} \sum_{E \subset K} w_E^K \delta_E
  (I_h(v_h^+)) \delta_E(I_h(v_h^+)) = \|\nabla
  I_h(v_h^+)\|_{L^2(\Omega)}^2.
\end{aligned}
$$ 
This proves that  
\begin{equation} \label{poisson-h}
(\nabla v_h, \nabla I_h(v_h^+)) \geq \|\nabla I_h(v_h^+)\|_{L^2(\Omega)}^2.
\end{equation}

We now finish the proof by induction. First, the result holds for $n=0$ by
assumption. Assume the result holds for $n-1$, i.e.
$\|u_h^{n-1}\|_{L^\infty(\Omega)} \leq 1$. Then, we define a special test
function $v_h \in V_h$ as $v_h := I_h \left( (u_h^{n}-1)^+ \right)$.
Notice that $\|u_h^{n-1}\|_{L^\infty(\Omega)} \leq 1$ implies 
$$ 
\frac{1}{k_n}(u_h^{n} - u_h^{n-1}) \geq
\frac{1}{k_n}(u_h^{n} - 1),
$$ 
which means that 
$$ 
\begin{aligned}
(\frac{1}{k_n}I_h\big((u_h^{n}-u_h^{n-1})v_h\big), 1) &=
\frac{1}{k_n}\int_{\Omega} I_h\big( (u_h^n - u_h^{n-1}) (u_h^n-1)^+
  \big) ~dx \\
&\geq \frac{1}{k_n}\int_\Omega I_h\big( (u_h^n-1)(u_h^n-1)^+ \big)
~dx =\frac{1}{k_n} \|I_h\big((u_h^n - 1)^+\big)\|_h^2. 
\end{aligned}
$$ 
Furthermore by \eqref{poisson-h} and the inductive assumption,
$$ 
\begin{aligned}
(\nabla u_h^n, \nabla v_h) &= (\nabla(u_h^n-1), \nabla
    I_h\big((u_h^n-1)^+\big)) \geq \|\nabla
I_h\big((u_h^n-1)^+\big)\|_{L^2(\Omega)}^2 \geq 0, \\
(I_h\big(f(u_h^n)v_h\big), 1) &= \int_\Omega I_h\big((u_h^n+1)u_h^n
    (u_h^n-1)(u_h^n-1)^+\big)~dx \geq 0.
\end{aligned}
$$ 
Therefore, 
$$ 
\frac{1}{k_n} \|I_h\big((u_h^n - 1)^+\big)\|_h^2 \leq 
(\frac{1}{k_n} I_h\big((u_h^{n}-u_h^{n-1})v_h\big), 1) + (\nabla
u_h^{n}, \nabla v_h) +
\frac{1}{\epsilon^2}(I_h\big(f(u_h^{n})v_h\big), 1) = 0,
$$ 
which implies $I_h\big((u_h^{n} - 1)^+\big) = 0$, thus $u_h^{n} \leq 1$.
Similarly, by choosing a special test function $v_h := I_h \left(
(u_h^{n}+1)^- \right)$, we can prove that $u_h^{n} \geq -1$.
Therefore, $\|u_h^{n}\|_{L^{\infty}(\Omega)} \leq 1$.
\end{proof}

%%In the test below, the discrete maximum principle for scheme
%%\eqref{fully implicit-lumping} is numerically verified.

\paragraph{Test 6} \label{test6} 
In this test, the same domain is chosen as in Test 1, and the random
initial condition for the Allen-Cahn equation is used with $\epsilon =
0.01$. In Figure \ref{fig:maximum-preserving}, it shows the random
initial condition, the evolutions, and the $L^{\infty}$-norm of the
numerical solutions at different time points.
%%We observe that the $L^{\infty}$-norm of the numerical solutions are
%%always less than or equal to $1$.
\begin{figure}[!htbp]
\centering 
\captionsetup{justification=centering}
\subfloat[$t=0$]{
\includegraphics[width=0.30\textwidth]{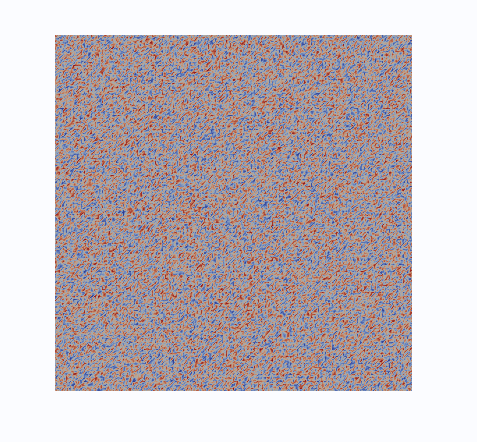}
}%
\subfloat[$t=0.003$]{
\includegraphics[width=0.30\textwidth]{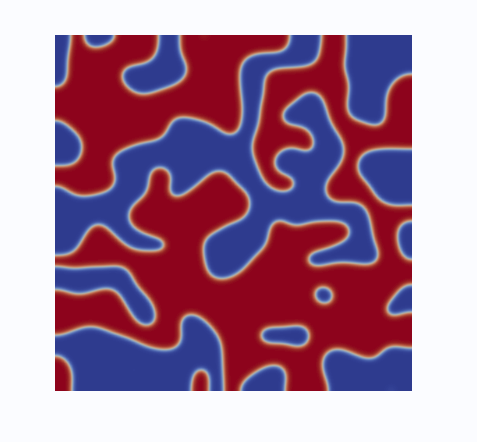}
}%
\subfloat[$L^\infty$-norm]{
\includegraphics[width=0.30\textwidth]{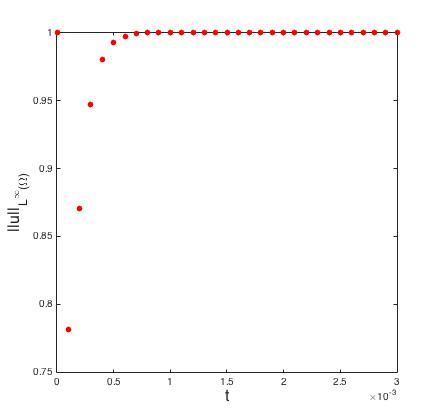}
}
\caption{Performance of modified FIS with random initial condition.} 
\label{fig:maximum-preserving}
\end{figure}

\begin{remark}
An analogous technique can be applied to prove the discrete maximum
principle for the convex splitting scheme with mass lumping 
\begin{equation} \label{convex-splitting-lumping}
(\frac{1}{k_n} I_h\big((u_h^{n}-u_h^{n-1})v_h\big), 1) + (\nabla
u_h^{n}, \nabla v_h) +
\frac{1}{\epsilon^2}(I_h\big([(u_h^n)^3-u_h^{n-1}]v_h\big), 1) = 0
\qquad \forall v_h \in
V_h. 
\end{equation}
This comes from the fact that \eqref{convex-splitting-lumping} can be
considered as the \eqref{fully-implicit-lumping} with the time step
size $\frac{\epsilon^2}{k_n+\epsilon^2}k_n$.
\end{remark}

\begin{remark}
We define the modified free-energy functional and discrete energy 
\begin{equation} \label{equ:FIS-AC-energy-lumping}
\begin{aligned}
J_{\epsilon,I}^{\rm AC}(u) &= \int_{\Omega} \frac{1}{2} |\nabla u|^2 +
\frac{1}{\epsilon^2} I_h(F(u)) ~dx, \\
E_{n,I}^{\rm AC}(u_h; u_h^{n-1}) &= J_{\epsilon,I}^{\rm AC}(u_h) + \frac{1}{2k_n}\int_\Omega
I_h(u_h-u_h^{n-1})^2~dx.
\end{aligned}
\end{equation}
We also define the   following energy minimization problem:  
\begin{equation} \label{equ:FIS-AC-energy-min-lumping}
u_h^n = \underset{u_h\in V_h}{\mathrm{argmin}} E_{n,I}^{\rm AC}(u_h;u_h^{n-1}).
\end{equation}
Similar to Theorem \ref{thm:convexity-FIS-AC}, we have the following
results:
\begin{enumerate}
\item Under the condition that $k_n \leq \epsilon^2$,  $E_{n,I}^{\rm AC}(\cdot;u_h^{n-1})$ is strictly convex on $V_h$.
\item The equation \eqref{convex-splitting-lumping} satisfies
$(E_{n,I}^{\rm AC})'(u_h^n; u_h^{n-1})(v_h) = 0$.
\item The following energy law holds
\begin{align}\label{energy-AC-lumping} 
J_{\epsilon,I}^{\rm AC}(u_h^n) + \frac{1}{2k_n}\|u_h^n -
u_h^{n-1}\|_{L^2(\Omega)}^2 \leq J_{\epsilon,I}^{\rm AC}(u_h^{n-1}).
\end{align}
\end{enumerate}
\end{remark}

%% preconditioner %% 
\subsection{A robust preconditioner for the Allen-Cahn equation}
Next we will analyze a simple preconditioner for the Newton
linearization of modified FIS \eqref{fully-implicit-lumping}.  With this
preconditioner, the resulting preconditioned conjugate gradient method
(PCG) significantly reduces the number of iterations of the conjugate
gradient method (CG), and moreover, the number of iterations is
uniform with respect to the spatial meshes which can be locally
refined. We acknowledge that some nonlinear multigrid methods have
been applied to numerical schemes similar to
\eqref{fully-implicit-lumping} in the literature, see
\cite{kim2004conservative,tai2002global,wise2007solving}.

We first define the mass lumping operator $\mathcal{I}_h[u]: V_h
\mapsto V_h$ as 
\begin{equation} \label{mass-lumping-operator}
(\mathcal{I}_h[u] v_h, w_h) := (I_h(uv_hw_h), 1) \qquad
\forall v_h, w_h \in V_h, u \in C(\bar{\Omega}).
\end{equation}
Let $\mathcal{I}_h = \mathcal{I}_h[1]$ for convenience.  The
Fr$\acute{e}$chet derivative of scheme \eqref{fully-implicit-lumping}
is denoted by $\mathcal{L}_{n,h}[u_h^n]: V_h \mapsto V_h$, such that  
\begin{equation} \label{frechet-lumping}
(\mathcal{L}_{n,h}[u_h^n] v_h, w_h) := (
\frac{1}{k_n}\mathcal{I}_h v_h, w_h ) - (\Delta_h
v_h, w_h) + \frac{1}{\epsilon^2}\left( \mathcal{I}_h[(3u_h^{n})^2 -
1]v_h, w_h\right) \quad \forall v_h, w_h \in V_h. 
\end{equation}

\begin{theorem}\label{thm:well-conditioned}
The upper and lower bounds of $\mathcal{L}_{n,h}[u_h^n]$ are given by
\begin{align}\label{eq:upper-lower} 
\frac{1-\gamma_n}{k_n}\mathcal{I}_h - \Delta_h \leq \mathcal{L}_{n,h}[u_h^n]
\leq \frac{1+2\gamma_n}{k_n} \mathcal{I}_h - \Delta_h.
\end{align}
where $\gamma_n := k_n\slash\epsilon^2$.
\end{theorem}
\begin{proof}
In light of \eqref{frechet-lumping}, we only need to prove 
$$ 
-\gamma_n (\mathcal{I}_h v_h, v_h) \leq
\frac{k_n}{\epsilon^2}(I_h\big([(3u_h^{n})^2 -
1](v_h)^2\big), 1) \leq 2\gamma_n (\mathcal{I}_h v_h, v_h) \qquad
\forall v_h \in V_h. 
$$ 
The left inequality can be proved by fact that $3(u_h^n)^2-1 \geq -1$,
and the right inequality can be proved by the fact that $3(u_h^n)^2 -
1 \leq 2$ due to the discrete maximum principle in Theorem
\ref{thm:maximum-AC-lumping}.
\end{proof}

Based on the Theorem \ref{thm:well-conditioned}, it is an immediate
consequence that when $\gamma_n \leq 1$, or $k_n \leq \epsilon^2$, $
(\mathcal{L}_{n,h}[u_h^n] v_h, v_h) \geq 0$ for any $v_h\in V_h$, which
implies the convexity of the discrete energy with mass lumping
$E_{n,I}^{\rm AC}(\cdot; u_h^{n-1})$ defined in \eqref{equ:FIS-AC-energy-lumping}.
Thus, the uniqueness and existence of FIS with mass lumping hold when
$k_n \leq \epsilon^2$.  Further, we can design a preconditioner for
$\mathcal{L}_{n,h}[u_h^n]$ as 
\begin{equation} \label{preconditioner-AC}
\mathcal{B}_{n,h} = \left(\frac{1-\gamma_n}{k_n}{\mathcal{I}}_h -
    {\Delta}_h\right)^{-1}.
\end{equation}
Then, we have the following theorem directly followed from
the Theorem \ref{thm:well-conditioned}.
\begin{theorem} \label{thm:condition-AC}
It holds that 
\begin{equation} \label{condition-AC}
\kappa(\mathcal{B}_{n,h}\mathcal{L}_{n,h}[u_h^n]) \leq
\frac{1+2\gamma_n}{1-\gamma_n}.
\end{equation}
\end{theorem}

\begin{remark}
When the uniform meshes are used with $h^{-1}=\mathcal
O(\epsilon^{-1})$ and $k_n =\mathcal O(\epsilon^2)$, then it is
apparent that $\mathcal{L}_{n,h}[u_h^n]$ is already well-conditioned. 
Therefore, the above Theorem~\ref{thm:condition-AC} is of special
interest when the adaptive meshes are used. 
\end{remark}

\paragraph{Test 7} \label{test7} 
In this test, consider the initial condition
\eqref{eq:initial-smooth2} and the scheme
\eqref{fully-implicit-lumping}, and $\epsilon = 0.02$,
$k_n=\frac{\epsilon^2}{2}=2\times10^{-4}$.  The simulation on adaptive
meshes is partially based on the MATLAB software package $i$FEM
\cite{chen2008iFEM}, and the mesh refining and coarsening are based on
the error estimator in \cite{feng2005posteriori}.  The adaptive
tolerance is $10^{-5}$ and the maximal bisection level $J=20$.  When
the maximal bisection level increases, the number of degrees of
freedom (DOF) increases, then the numbers of iterations of CG and PCG
are compared in the Table \ref{tab1} to verify the theoretical
results.

\begin{table}[!htbp]
\begin{center}
\centering
\begin{tabular}{|l|c|c|c|c|c|c|c|c|c|c|c|c|c|c|}
\hline
DOF & 301 & 368 & 430 & 510 & 566 & 672 & 1276 & 1633 & 2044 & 2535 &
3217 & 4027 & 4610  \\ \hline
CG & 21 & 32 & 37& 38 & 41 & 45 & 58 & 61 & 68 & 78 & 96 & 106 & 117\\ \hline
PCG & 9  & 8 & 8 & 9 & 8 & 8 & 8 & 8 & 8 & 8 & 8 & 8 & 8\\ \hline
\end{tabular}
\smallskip
\caption{The number of iterations for CG and PCG.} 
\label{tab1} 
\end{center}
\end{table}

\section{Second-order schemes}
\label{sec:2nd-order}
In this section, we shall consider the second-order schemes.  

\subsection{(Modified) Crank-Nicolson scheme for the Allen-Cahn equation}
The standard Crank-Nicolson scheme for the Allen-Cahn equation, is to
seek $u_h^n\in V_h$ for $n = 1, 2, \cdots$, such that
\begin{align}\label{eq:SCN-AC}
\bigl(\frac{ u_h^{n}- u_h^{n-1}}{k_n},v_h\bigr)+\bigl(\frac{\nabla
    u_h^{n} + \nabla u_h^{n-1}}{2}, \nabla v_h\bigr)
+\frac{1}{2\epsilon^2}(f(u_h^n)+f(u_h^{n-1}), v_h) = 0 \qquad \forall
v_h\in V_h.
\end{align}
Although the standard Crank-Nicolson scheme can not be proved
energy-stable, in view of \eqref{eq:SCN-AC}, we can still show its
convexity by defining the following discrete energy 
\begin{equation}
\label{eq:SCN-AC-discrete-energy}
E_{n,\rm CN}^{\rm AC}(u_h; u_h^{n-1}) = \frac{1}{2}
\left\|\frac{\nabla u_h+\nabla u_h^{n-1}}{2}\right\|_{L^2(\Omega)}^2 +
\frac{1}{4k_n}\|u_h-u_h^{n-1}\|^2_{L^2(\Omega)}
+ \frac{1}{4\epsilon^2}\int_{\Omega} F(u_h) + f(u_h^{n-1})u_h~dx.
\end{equation}
\begin{theorem}\label{thm:convexity-SCN}
Under the condition that $k_n \leq 2\epsilon^2$, we have  
\begin{enumerate}
\item $E_{n,\rm CN}^{\rm AC}(\cdot; u_h^{n-1})$ is strictly convex on $V_h$;
\item The solution of the modified Crank-Nicolson scheme
\eqref{eq:SCN-AC} satisfies
 \begin{align}\notag
 u_h^{n}&= \underset{u_h\in V_h}{\mathrm{argmin}}\, E_{n,\rm CN}^{\rm AC}(u_h; u_h^{n-1}),
\end{align}
which is uniquely solvable.
%\item The following energy law holds
%\begin{align}\label{stand_CN_energy} 
%J_\epsilon^{AC}(u_h^n) + \frac{1}{2k}\|u_h^n -
%u_h^{n-1}\|_{L^2(\Omega)}^2 \leq J_\epsilon^{AC}(u_h^{n-1}).
%\end{align}
\end{enumerate}
\end{theorem}
\begin{proof}
A direct calculation shows that
\begin{equation}
\label{convexity-SCN}
(E_{n,\rm CN}^{\rm AC})''(u_h; u_h^{n-1})(v_h,v_h)= \frac14\|\nabla
 v_h\|_{L^2(\Omega)}^2
 +(\frac{1}{2k_n}-\frac{1}{4\epsilon^2})\|v_h\|^2_{L^2(\Omega)}
+ \frac{1}{4\epsilon^2}\int_{\Omega}3u_h^2v_h^2~dx.\nonumber
\end{equation}
This implies that $E_{n,\rm CN}^{\rm AC}(\cdot; u_h^{n-1})$ is a strictly
convex functional when $k_n \leq 2\epsilon^2$. The rest of the proof is
standard.
\end{proof}

With the purpose of energy stability, the {\it modified Crank-Nicolson
scheme} \cite{du1991numerical, shen2010numerical,
condette2011spectral} is constructed as follows: Find $u_h^n\in V_h$
for $n = 1, 2, \cdots$, such that
\begin{align}\label{eq:CN-AC}
\bigl(\frac{ u_h^{n}- u_h^{n-1}}{k_n},v_h\bigr)+\bigl(\frac{\nabla
    u_h^{n} + \nabla u_h^{n-1}}{2}, \nabla v_h\bigr)
+\frac{1}{\epsilon^2}(\tilde F[u_h^n, u_h^{n-1}], v_h) = 0 \qquad \forall
v_h\in V_h,
\end{align}
where
$$
\tilde F[u, u_h^{n-1}]=
\begin{cases}
\frac{F(u)-F(u_h^{n-1})}{ u - u_h^{n-1}} & u\neq u^{n-1}_h,\\
u^3 - u & u =  u^{n-1}_h.
\end{cases}
$$

\begin{lemma}[\cite{shen2010numerical,condette2011spectral}]
\label{thm:unconditionally-CSS}
The modified Crank-Nicolson scheme \eqref{eq:CN-AC} is
unconditionally energy stable.  More precisely, for any $k_n>0$,
\begin{align}\label{eq:CN-AC-energy}
J_\epsilon^{\rm AC}( u_h^{n}) + \frac{1}{k_n}\|u_h^{n} -
u_h^{n-1}\|_{L^2(\Omega)}^2 = J_\epsilon^{\rm AC}( u_h^{n-1}).
\end{align}
\end{lemma}
\begin{proof}
\eqref{eq:CN-AC-energy} is an immediate consequence by taking
$v_h=u_h^{n} - u_h^{n-1}$ in \eqref{eq:CN-AC}. 
\end{proof}

The modified Crank-Nicolson scheme \eqref{eq:CN-AC} is
unconditionally energy-stable but it is not unconditionally convex as
we shall see below. In view of \eqref{eq:CN-AC}, we define the
following discrete energy 

\begin{equation}
\label{eq:CN-AC-discrete-energy}
E_{n,\rm MCN}^{\rm AC}(u_h; u_h^{n-1}) = \frac12\left\| \frac{\nabla u_h +
  \nabla u_h^{n-1}}{2}\right\|_{L^2(\Omega)}^2 +
  \frac{1}{4k_n}\|u_h-u_h^{n-1}\|^2_{L^2(\Omega)} +
  \frac{1}{2\epsilon^2}\int_{\Omega}\check{G}(u_h; u_h^{n-1})~dx, \\
\end{equation}
where $\check{G}(u_h; u_h^{n-1}) = \check{G}_+(u_h; u_h^{n-1}) -
\check{G}_-(u_h;u_h^{n-1})$, and
$$
\check{G}_+(u_h; u_h^{n-1})=\frac{1}{4}\bigl[ \frac{1}{4}u_h^4 +
\frac{u_h^{n-1}}{3} u_h^3 + \frac{(u_h^{n-1})^2}{2} u_h^2 + (u_h^{n-1})^3u_h
\bigr]\quad \text{and} \quad 
\check{G}_-(u_h; u_h^{n-1})=
\frac14u_h^2+\frac12u_hu_h^{n-1}. 
$$
\begin{theorem}\label{thm:convexity-CN}
Under the condition that $k \leq 2\epsilon^2$, we have  
\begin{enumerate}
\item $E_{n,\rm MCN}^{\rm AC}(\cdot;u_h^{n-1})$ is strictly convex on $V_h$;
\item The solution of the modified Crank-Nicolson scheme
\eqref{eq:CN-AC} satisfies
 \begin{align}\notag
 u_h^{n}&= \underset{u_h\in V_h}{\mathrm{argmin}}\, E_{n,\rm MCN}^{\rm AC}(u_h; u_h^{n-1}),
\end{align}
which is uniquely solvable.
\end{enumerate}
\end{theorem}
\begin{proof}
A direct calculation shows that
\begin{equation}
\label{convexity-MCN}
\begin{aligned}
 (E_{n,\rm MCN}^{\rm AC})''(u_h; u_h^{n-1})(v_h,v_h) &= \frac14\|\nabla v_h\|_{L^2(\Omega)}^2 +(\frac{1}{2k_n}-\frac{1}{4\epsilon^2})\|v_h\|^2_{L^2(\Omega)} \\
&+ \frac{1}{8\epsilon^2}\int_{\Omega}\bigl[ 3u_h^2+ 2u_h^{n-1}u_h
  + (u_h^{n-1})^2 \bigr]v_h^2~dx.
  \end{aligned}
\end{equation}
This implies that $E_{n, \rm MCN}^{\rm AC}(\cdot; u_h^{n-1})$ is a strictly convex
functional when $k_n\leq 2\epsilon^2$.  The rest of the proof is
standard.
\end{proof}
The ``convexity size'' of standard and modified Crank-Nicolson schemes
are the same. We also observe the similar numerical performance of
these two schemes (see Test 8, 9 and 11 below), although the standard
Crank-Nicolson does not satisfy the energy stability. 

\begin{remark}
Similar to the CSS \eqref{css}, we can obtain the corresponding
convex splitting version of the modified Crank-Nicolson scheme in the
following:
\begin{equation} \label{CSS-MCN}
(\frac{ u_h^{n}- u_h^{n-1}}{k_n},v_h)+(\frac{\nabla u_h^{n} +
\nabla u_h^{n-1}}{2}, \nabla v_h)
+\frac{1}{\epsilon^2} \left(g_+(u_h^{n}; u_h^{n-1})- g_{-}(u_h^{n-1};
      u_h^{n-1}), v_h\right) = 0\quad \forall v_h\in V_h,
\end{equation}
where 
$$
g_+(u_h; u_h^{n-1})=G_+'(u_h;u_h^{n-1})=\frac14\left[ u_h^3+ u_h^{n-1}u_h^2 + (u_h^{n-1})^2
u_h + (u_h^{n-1})^3 \right],
$$
$$
g_-(u_h; u_h^{n-1})=G_-'(u_h; u_h^{n-1}) = \frac12( u_h+ u_h^{n-1}).
$$
Similar to Theorem~\ref{thm:css-fis}, we know that the convex
splitting scheme~\eqref{CSS-MCN} can be recast as the modified
Crank-Nicolson scheme \eqref{eq:CN-AC} with the time step size 
$k_n' = \frac{2\epsilon^2}{k_n+2\epsilon^2}k_n$. This also shows the delay
effect of the convex splitting scheme \eqref{CSS-MCN} to the original
fully implicit scheme \eqref{eq:CN-AC}, but with a slightly
different delay-factor: $ \delta_n=\frac{2\epsilon^2}{k_n+2\epsilon^2}$.

Again, similar to the argument we made in \S~\ref{subsec:CSS-AC}, the
convex splitting scheme~\eqref{CSS-MCN} derived here is the same as
the original modified Crank-Nicolson scheme~\eqref{eq:CN-AC} in
disguise with a reduced time step size.
\end{remark}

\paragraph{Test 8} \label{test8a} 
In this simulation, we minimize the discrete energy
\eqref{eq:SCN-AC-discrete-energy} for the Allen-Cahn equation at each time
step. The computational domain is $\Omega=(-1,1)^2$, and
parameter is $\epsilon =5\times 10^{-3}$.  
In order to smooth the initial value, we first compute the solution
from $t = 0$ to $t = 0.01$ with $k = 10^{-3}$, namely 
$ k_n = 10^{-5}$ for $n = 1, 2, \cdots, 10.$
Then, we switch to $k_{11}=10^{-2}$.
After the smoothing the random initial value, we test the
dependency on the initial guess for the L-BFGS minimization algorithm. 
Figure
\ref{fig:ck-regular-global-local-minimizer} shows different results with
different initial guess for $u$ and using the standard Crank-Nicolson scheme. 
We observe that the result with the lowest energy is the one the closest to the reference solution.

\begin{figure}[!htbp]
\centering 
\subfloat[Reference solution at $t=0.02$]{
  \includegraphics[scale=0.25]{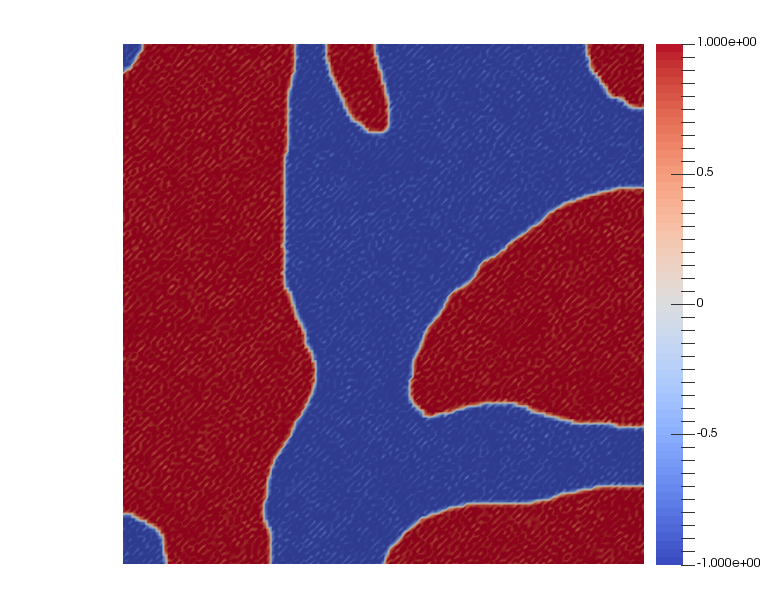} 
  \label{fig:regular-global-local-refb}
}%
\subfloat[$u_h^{10}$ as initial guess, $E_{1,\rm MCN}^{\rm AC} = 8.338$]{
  \includegraphics[scale=0.25]{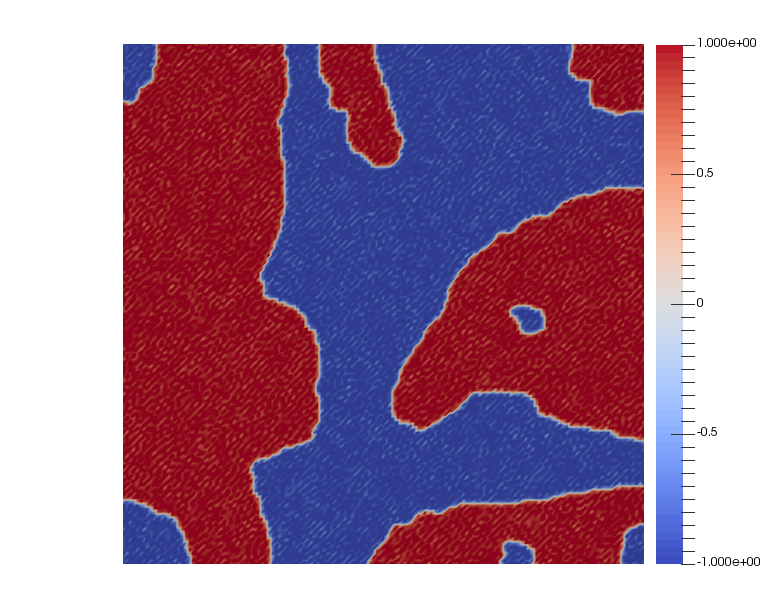} 
  \label{fig:regular-global-local-globalb}
}\\
\subfloat[Random initial guess, $E_{1,\rm MCN}^{\rm AC} = 13.017$]{
  \includegraphics[scale=0.25]{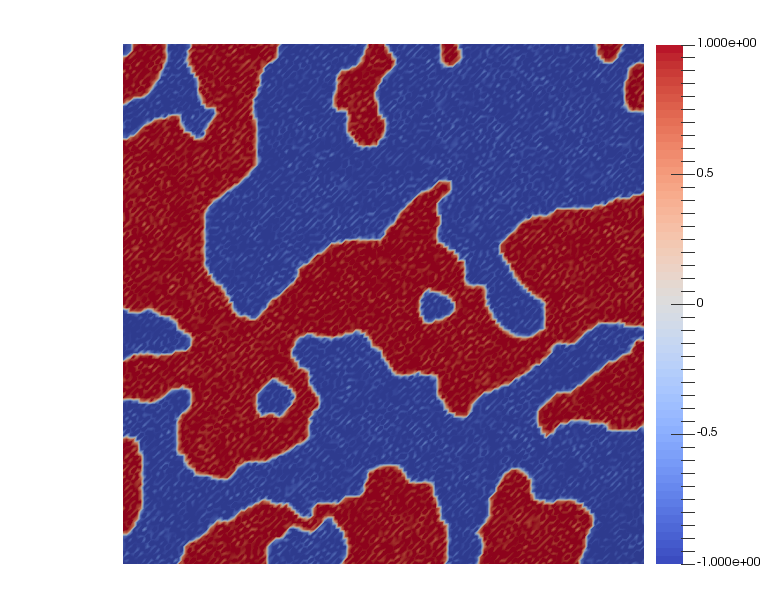} 
  \label{fig:regular-global-local-local1b}
}
\subfloat[Random initial guess, $E_{1,\rm MCN}^{\rm AC} = 12.569$]{
  \includegraphics[scale=0.25]{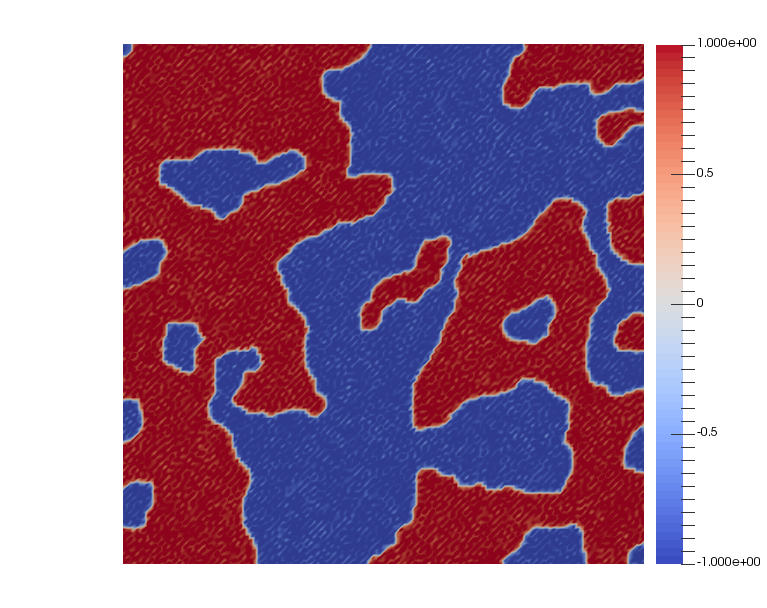} 
  \label{fig:regular-global-local-local2b}
}\\
\caption{The Allen-Cahn equation with random initial value and using standard Crank-Nicolson: Minimizers
at $t=1.1\times 10^{-2}$ for different initial guesses in the L-BFGS
algorithm.}
\label{fig:ck-regular-global-local-minimizer}
\end{figure}

\paragraph{Test 9} \label{test8} 
In this simulation, we minimize the discrete energy
\eqref{eq:CN-AC-discrete-energy} for the Allen-Cahn equation at each time
step. The computational domain and
parameter are the same as Test 8.  Figure
\ref{fig:ck-global-local-minimizer} shows different results with
random initial $u$ and using the modified Crank-Nicolson scheme.  Even
though any solution given by the modified Crank-Nicolson is
unconditionally energy stable, we observe that the result with the
lowest energy is the one the closest to the reference solution.
Moreover, the unconditionally stable scheme (e.g. modified
Crank-Nicolson) can not guarantee the physical solution. 

\begin{figure}[!htbp]
\centering 
\subfloat[Reference solution at $t=0.01$]{
  \includegraphics[scale=0.25]{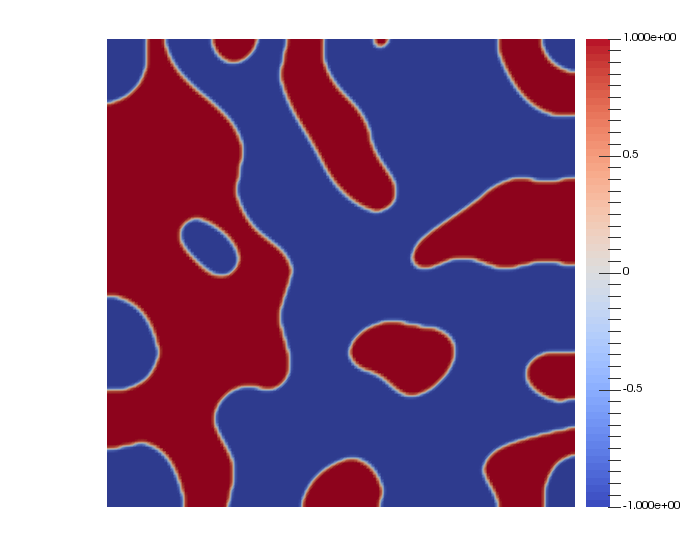} 
  \label{fig:global-local-refb}
}%
\subfloat[$u_h^{0}$ as initial guess, $E_{1,\rm MCN}^{\rm AC} = -297.176$]{
  \includegraphics[scale=0.25]{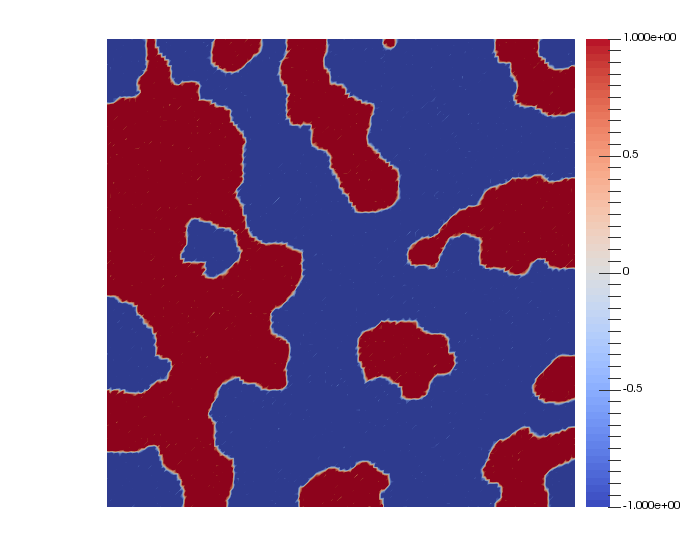} 
  \label{fig:global-local-globalb}
}\\
\subfloat[Random initial guess, $E_{1,\rm MCN}^{\rm AC} =-284.995$]{
  \includegraphics[scale=0.25]{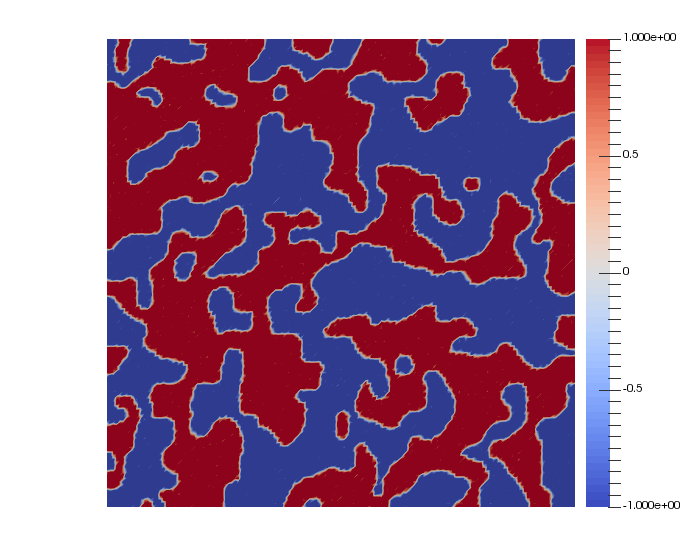} 
  \label{fig:global-local-local1b}
}
\subfloat[Random initial guess, $E_{1,\rm MCN}^{\rm AC} =-287.473$]{
  \includegraphics[scale=0.12]{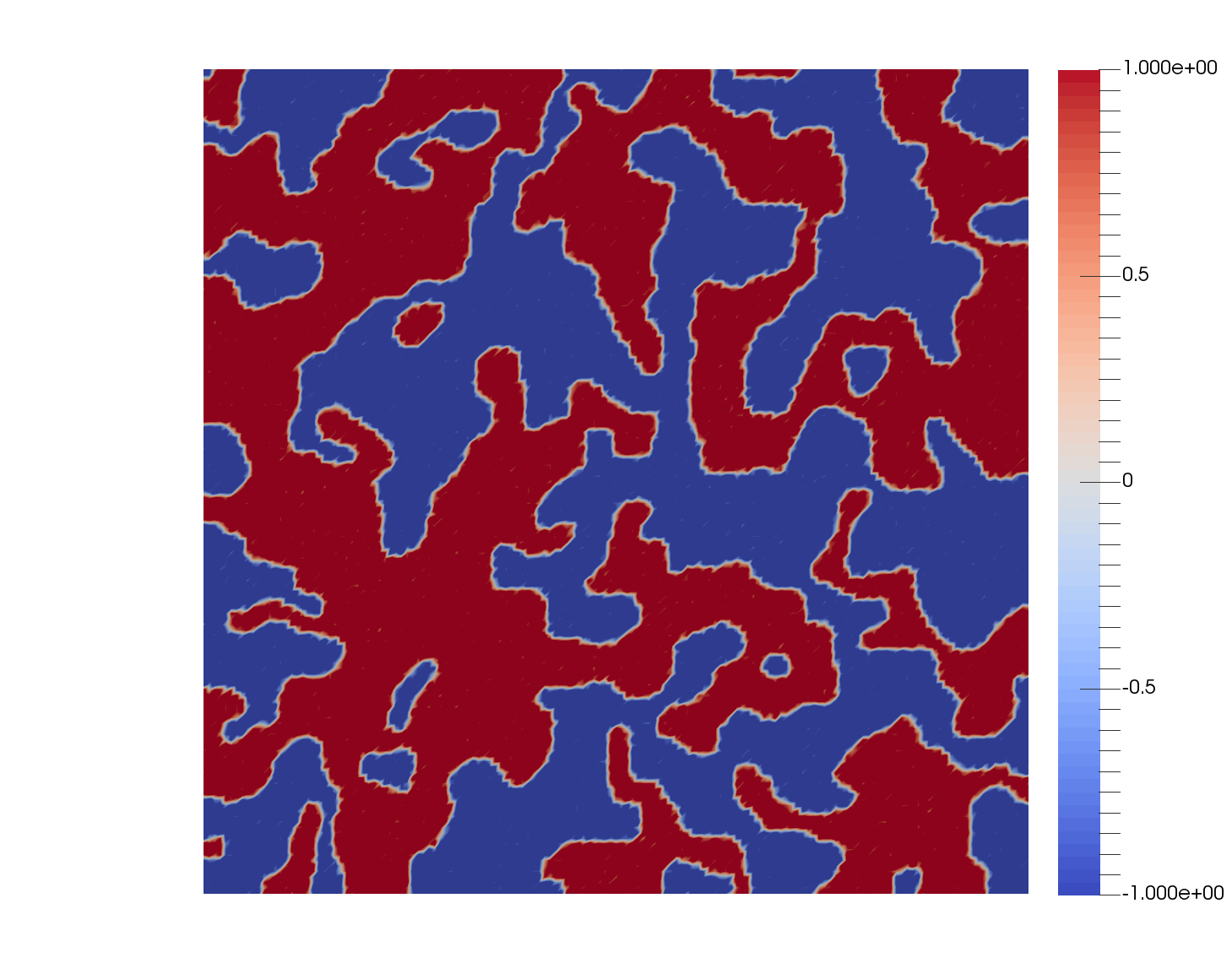} 
  \label{fig:global-local-local2b}
}\\
\caption{The Allen-Cahn equation with random initial value and using modified Crank-Nicolson: Minimizers
at $t=0.01$ for different initial guesses in the L-BFGS algorithm. (Here we add a constant to the discrete energy 
\eqref{eq:CN-AC-discrete-energy}, which does not affect the minimizers.)}
\label{fig:ck-global-local-minimizer}
\end{figure}

\paragraph{Test 10} \label{test9} 
Next, as done in the previous section, we evolve the Allen-Cahn
equation with different time step sizes to
see the two phases regroup. Three different computations with $k_n =
10^{-5}$ (convex case), $k_n = 10^{-4}$ and $k_n = 10^{-3}$
(non-convex cases) are considered. In Figure \ref{fig:ck-2} shows the
random initial value and the evolutions of the numerical solutions at
different $t$'s, using the modified Crank-Nicolson for time
discretization. It can be observed that the solutions in all these
cases behave similarly. Furthermore,    the evolutions of the physical
energies, see Figure \ref{fig:ck-3}, shows the energy-stability of the
energy minimization version of the modified Crank-Nicolson scheme,
which is in agreement with the Theorem \ref{thm:convexity-SCN}.

\begin{figure}[!htbp]
\centering 
\subfloat[$t=0$, $k_n = 10^{-5}$]{
  \includegraphics[scale=0.17]{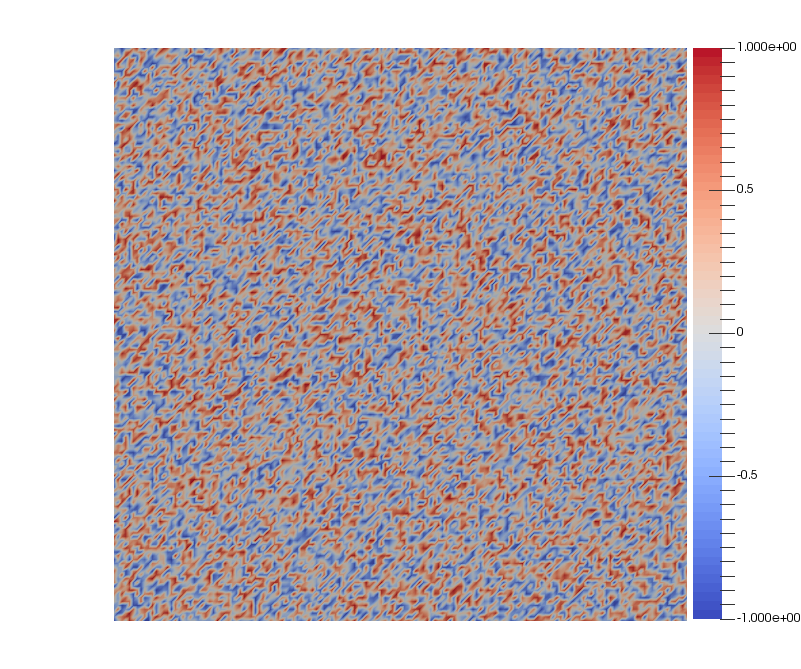}
}%
\subfloat[$t=0$, $k_n = 10^{-4}$]{
  \includegraphics[scale=0.17]{ck_1e-5_t=0.png}
}%
\subfloat[$t=0$, $k_n = 10^{-3}$]{
  \includegraphics[scale=0.17]{ck_1e-5_t=0.png}
} 
\\
\subfloat[$t=0.02$, $k_n = 10^{-5}$]{
  \includegraphics[scale=0.17]{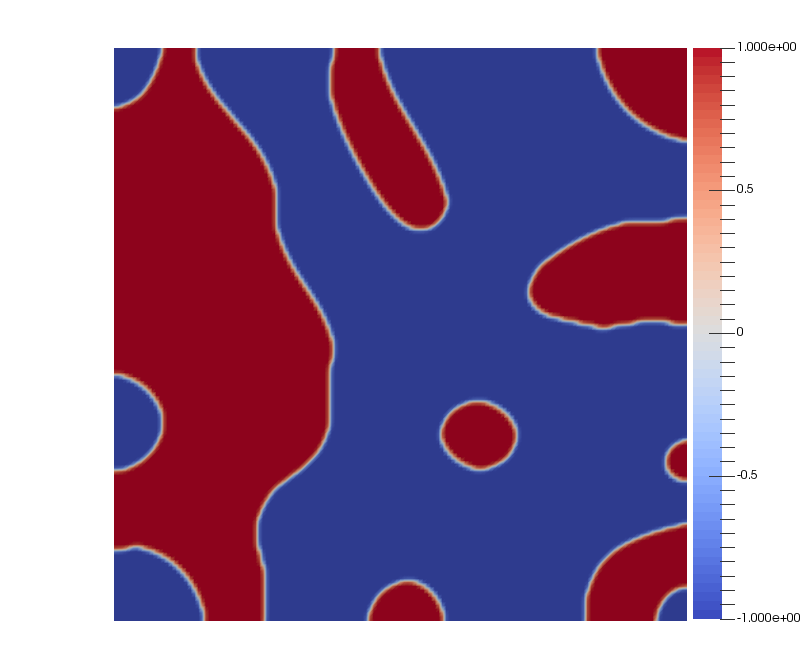}
} 
\subfloat[$t=0.02$, $k_n = 10^{-4}$]{
  \includegraphics[scale=0.17]{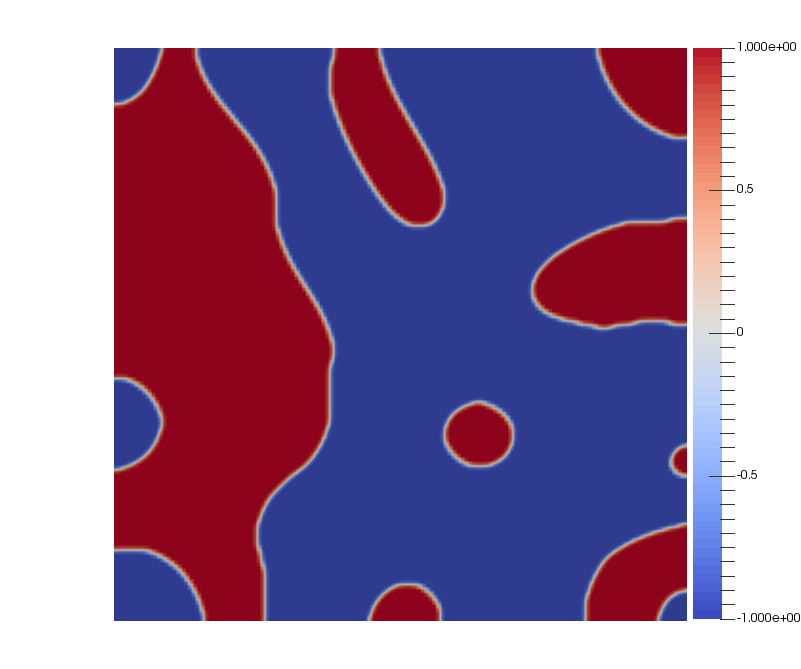}
}%
\subfloat[$t=0.02$, $k_n = 10^{-3}$]{
  \includegraphics[scale=0.17]{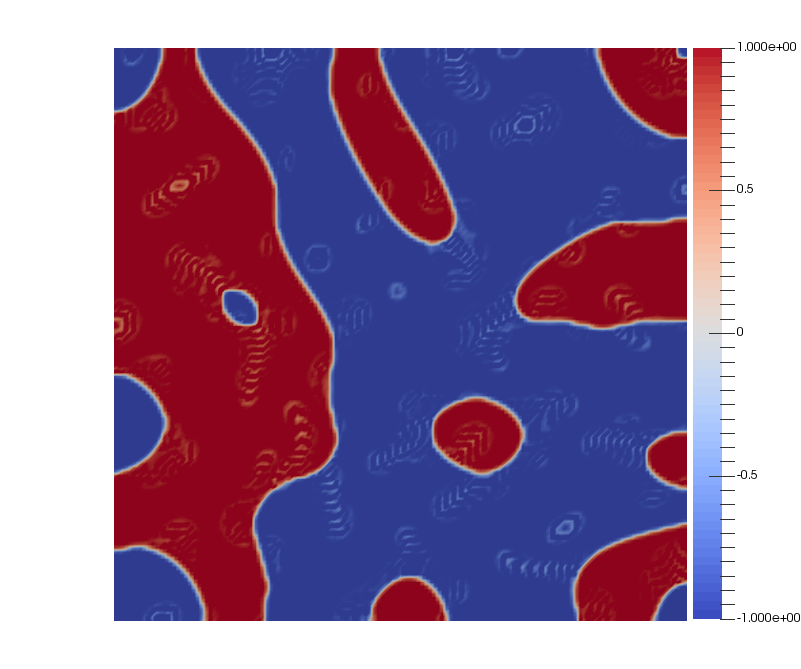} 
}%
\\
\subfloat[$t=0.05$, $k_n = 10^{-5}$]{
  \includegraphics[scale=0.17]{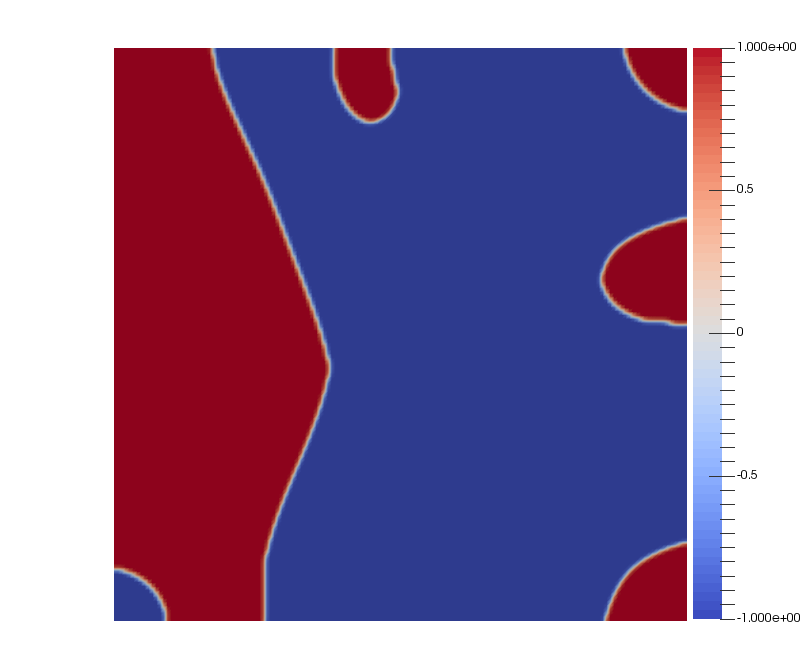}
} 
\subfloat[$t=0.05$, $k_n = 10^{-4}$]{
  \includegraphics[scale=0.17]{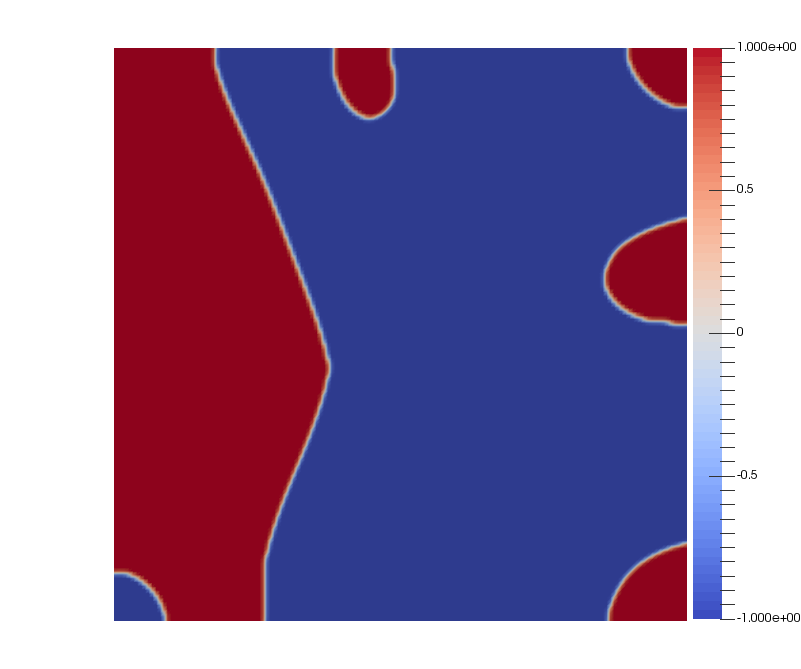}
}%
\subfloat[$t=0.05$, $k_n = 10^{-3}$]{
  \includegraphics[scale=0.17]{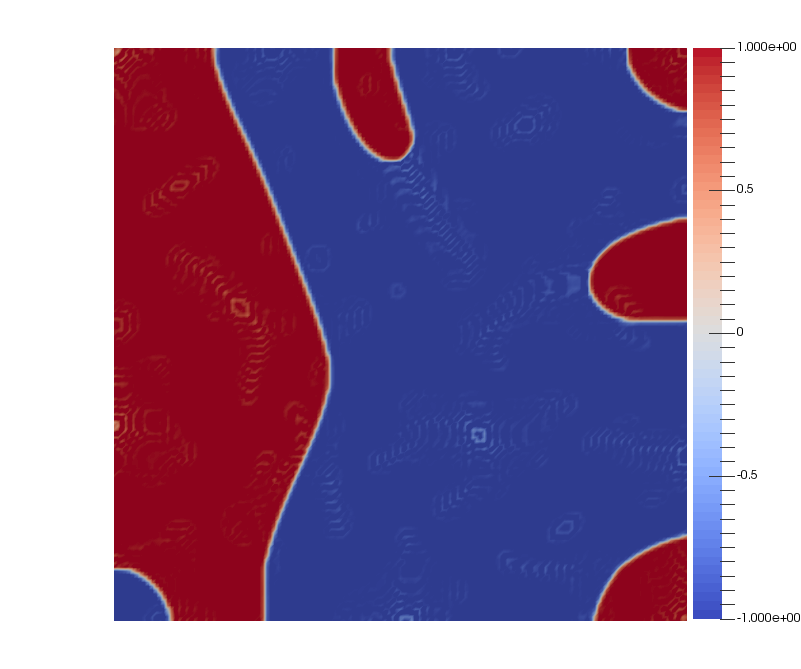}
}%
\\
\subfloat[$t=0.14$, $k_n = 10^{-5}$]{
  \includegraphics[scale=0.17]{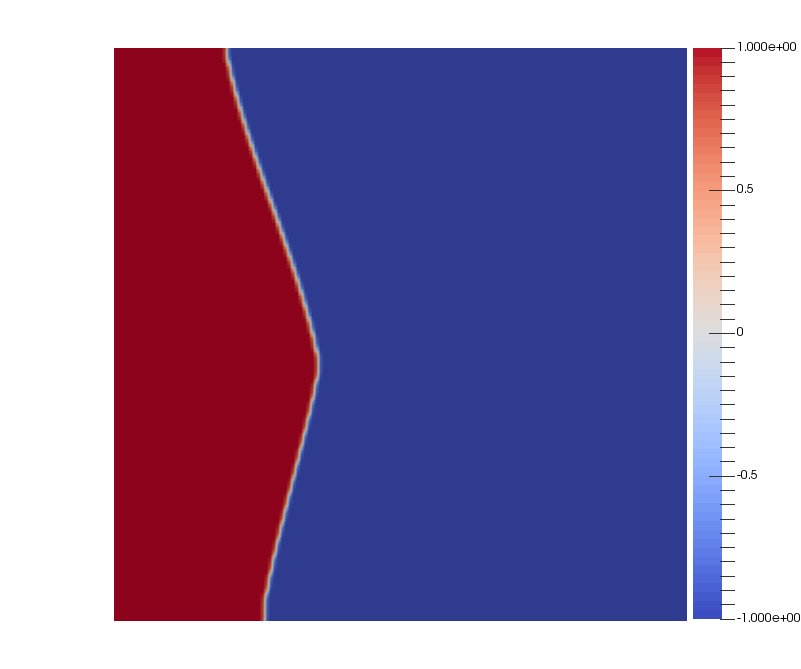}
} 
\subfloat[$t=0.14$, $k_n = 10^{-4}$]{
  \includegraphics[scale=0.17]{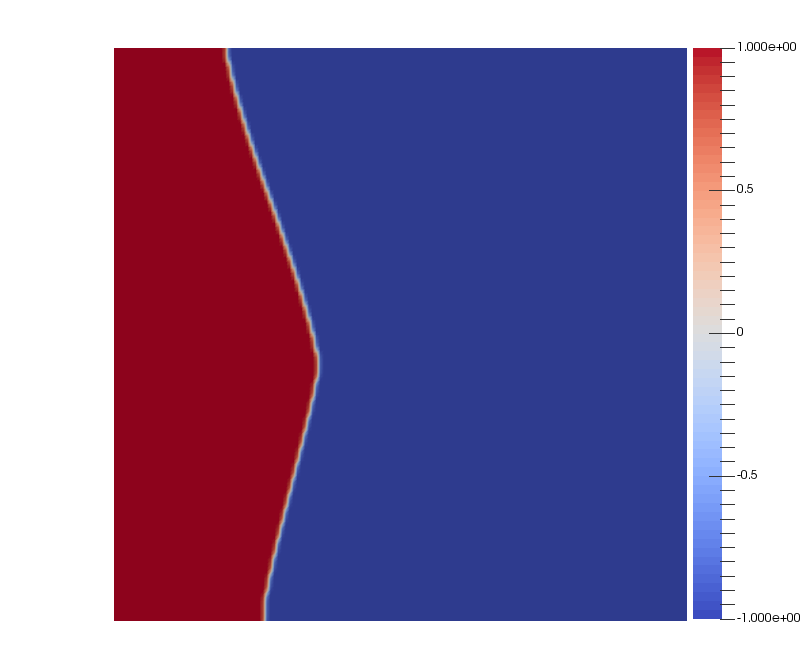}
}%
\subfloat[$t=0.14$, $k_n = 10^{-3}$]{
  \includegraphics[scale=0.17]{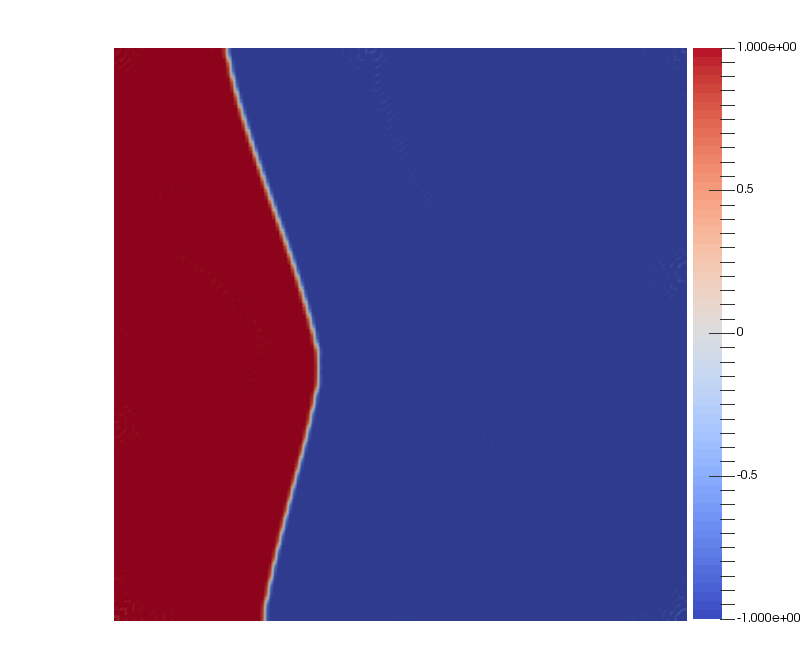}
}%
\\
%%\end{center}
\caption{The Allen-Cahn with random initial value: Plot of the
  solutions at different $t$'s.}
\label{fig:ck-2}
\end{figure}

\begin{figure}[!htbp]
\begin{center}
 \includegraphics[scale=0.6]{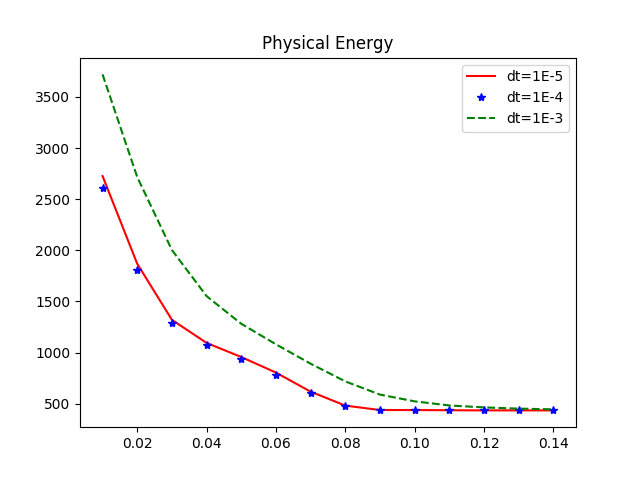}
\end{center}
\caption{The Allen-Cahn equation with random initial guess: Evolutions
  of physical energies.}
\label{fig:ck-3}
\end{figure}

\subsection{Modified Crank-Nicolson scheme for the Cahn-Hilliard equation}
The modified Crank-Nicolson scheme \cite{du1991numerical,
shen2010numerical, condette2011spectral} for the Cahn-Hilliard model
is defined as follows: Find $u_h^n \in V_h,\ w_h^n \in V_h$ for
$n=1,2,\cdots$, such that 
\begin{equation}\label{eq:CN-CH}
\begin{aligned}
(\frac{u_h^{n}-u_h^{n-1}}{k_n},\eta_h)+(\nabla w_h^{n},\nabla \eta_h) &=0
\qquad \forall \,\eta_h\in V_h, \\
\epsilon (\frac{\nabla u_h^n + \nabla u_h^{n-1}}{2},\nabla v_h) +
\frac{1}{\epsilon}(\tilde F[u_h^n,
u_h^{n-1}],v_h)-(w_h^{n},v_h) &=0 \qquad \forall\, v_h\in V_h. 
\end{aligned}
\end{equation}

\begin{lemma}[\cite{condette2011spectral, shen2010numerical}]
\label{lem_CR_CH_energy}
The modified Crank-Nicolson scheme \eqref{eq:CN-CH} is
unconditionally energy stable.  More precisely, for any $k_n>0$,
\begin{align}\label{eq_CR_CH_energy}
J_\epsilon^{\rm CH}( u_h^{n}) + \frac{1}{k_n}\|\nabla\Delta_h^{-1}(u_h^{n} -
u_h^{n-1})\|_{L^2(\Omega)}^2 = J_\epsilon^{\rm CH}( u_h^{n-1}).
\end{align}
\end{lemma}
\begin{proof}
It can be directly proved by taking
$\eta_h=\Delta_h^{-1}(u_h^n-u_h^{n-1})$ and $v_h=u_h^n-u_h^{n-1}$ in
\eqref{eq:CN-CH}.  
\end{proof}

Consider the Cahn-Hilliard equation, we define the following discrete energy 
\begin{equation}\label{eq_CR_CH4}
E_{n,\rm MCN}^{\rm CH}(\theta_h; u_h^{n-1}) = \frac{\epsilon}{2}\left\| \frac{\nabla
\theta_h + \nabla u_h^{n-1}}{2}\right\|_{L^2(\Omega)}^2 + 
\frac{1}{4k_n}\|\nabla\Delta^{-1}_h(\theta_h-u_h^{n-1})\|^2_{L^2(\Omega)}
+ \frac{1}{2\epsilon}\int_{\Omega}\check{G}(\theta_h,u_h^{n-1})~dx.
\end{equation}

\begin{theorem}\label{thm_CR_CH1}
Under the assumption that $k_n \leq 8\epsilon^3$, we have  
\begin{enumerate}
\item $E_{n,\rm MCN}^{\rm CH}(\cdot;u_h^{n-1})$ is strictly convex on $\mathring{V}_h$;
\item The solution of the modified Crank-Nicolson scheme
\eqref{eq:CN-CH} satisfies
 \begin{align}\notag
u_h^{n}=u_h^{n-1}+\theta_h^n, \mbox{ with }
\theta_h^n= \underset{\theta_h\in \mathring{V}_h}{\mathrm{argmin}}\,E_{n,\rm MCN}^{\rm CH}(\theta_h; u_h^{n-1}),
\end{align}
which is uniquely solvable.
\end{enumerate}
\end{theorem}
\begin{proof}
By the definition of operator $\Delta_h$ and the Schwarz's inequality, we have
\begin{align}\label{eq_laplace_inverse}
\frac{1}{2\epsilon}\|v_h\|_{L^2(\Omega)}^2\le \frac{1}{8\epsilon^3}\|\nabla\Delta_h^{-1}v_h\|^2_{L^2(\Omega)} + \frac{\epsilon}{2}\|\nabla v_h\|_{L^2(\Omega)}^2.
\end{align}
A direct calculation shows that
\begin{align} \label{convexity-MCN2}
 (E_{n,\rm MCH}^{\rm CH})''(\theta_h; u_h^{n-1})(v_h,v_h)&= \frac{\epsilon}{4}\|\nabla v_h\|_{L^2(\Omega)}^2 +\frac{1}{2k_n}\|\nabla\Delta_h^{-1}v_h\|^2_{L^2(\Omega)}-\frac{1}{4\epsilon}\|v_h\|^2_{L^2(\Omega)}\\
&\qquad+ \frac{1}{2\epsilon}\int_{\Omega}\bigl[ 3\theta_h^2+ 2u_h^{n-1}\theta_h + (u_h^{n-1})^2 \bigr]v_h^2~dx.\nonumber
\end{align}
This implies that $E_{n,\rm MCN}^{\rm CH}(\cdot; u_h^{n-1})$ is strictly convex
when $k_n\leq 8\epsilon^3$.  The rest of the proof is standard.
\end{proof}
\begin{remark}
Similar to the Allen-Cahn equation, the standard Crank-Nicolson
can also be constructed and analyzed for the Cahn-Hilliard equations.  
\end{remark}

\subsection{Some other second-order partially implicit schemes}
In this section, we briefly discuss several other second-order
partially implicit schemes. 

% BDF2
{\it Second-order stabilized semi-implicit scheme (BDF2)}: Seeking
$u_h^n\in V_h$ for $n=1,2,\cdots$, such that
\begin{equation}\label{eq:BDF2-AC}
\begin{aligned}
&(\frac{3u_h^{n}-4u_h^{n-1}+u_h^{n-2}}{2k_n},v_h)+(\nabla
u_h^n,\nabla v_h) +
\frac{1}{\epsilon^2}((2f(u_h^{n-1})-f(u_h^{n-2})),v_h)\\
&\qquad+ 
\frac{S}{\epsilon^2}(u_h^n-2u_h^{n-1}+u_h^{n-2},v_h)=0 \qquad\forall
v_h\in V_h,
\end{aligned}
\end{equation}
where $S>0$ (set as $S=10$ in the Test 10) is a stabilized constant.

% Second-order CSS 
{\it Second-order convex splitting scheme (CSS2)}: Seeking $u_h^n \in
V_h$ for $n = 1,2,\cdots$, such that 
\begin{equation} \label{eq:CSS2}
(\frac{u_h^n - u_h^{n-1}}{k_n}, v_h) + (\frac{\nabla u_h^n + \nabla
    u_h^{n-1}}{2}, \nabla v_h) +
\frac{1}{\epsilon^2}(g_+(u_h^{n}, u_h^{n-1}) -
    \frac{1}{2\epsilon^2}(3u_h^{n-1} - u_h^{n-2}), v_h) = 0.
\end{equation}

We know that BDF2 is a linear scheme so that satisfies the convexity
property.  Similar to the argument for the CSS version of modified
Crank-Nicolson scheme \eqref{CSS-MCN}, we know that \eqref{eq:CSS2}
also satisfies the convexity property. When $k_n\leq \epsilon^2$, these
second-order splitting schemes perform well (see Test 11 below).
However, we observe the following phenomenon for these second-order
splitting schemes:
\begin{enumerate}
\item They do not satisfy the discrete maximum principle, and it is
frequently worse than the first-order scheme;
\item They still suffer the lagging phenomenon or delayed convergence
for large time step size (see Test 11 below);  
\end{enumerate}

\paragraph{Test 11} \label{test11} 
In this test, the same domain and initial conditions are chosen as in
Test 1. Figure \ref{fig:2nd-CSS-small} shows the evolution of the
radius with respect to time for different second-order schemes. We
observe that all these second-order schemes perform well when
$k_n=\epsilon^2$. The performance of standard and modified
Crank-Nicolson schemes are similar.  When increasing the time step
size, however, we observe that the lagging phenomenon exists for the
CSS2 (see Figure \ref{fig:2nd-CSS-large}).

\begin{figure}[!htbp]
\centering 
\captionsetup{justification=centering}
\subfloat[Small time step: $k_n=\epsilon^2 = 0.0004$]{
  \includegraphics[width=0.35\textwidth]{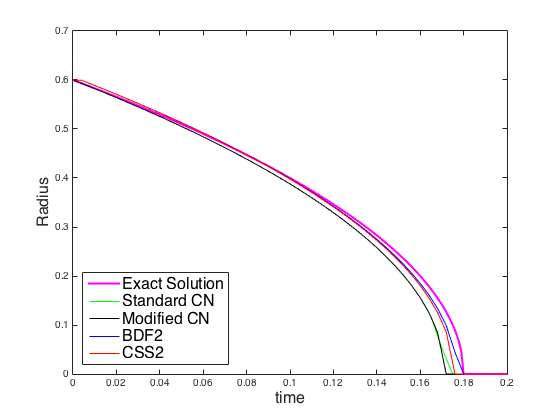}
  \label{fig:2nd-CSS-small}
}%
\subfloat[Large time step: $k_n=15\epsilon^2 = 0.006$]{
\includegraphics[width=0.35\textwidth]{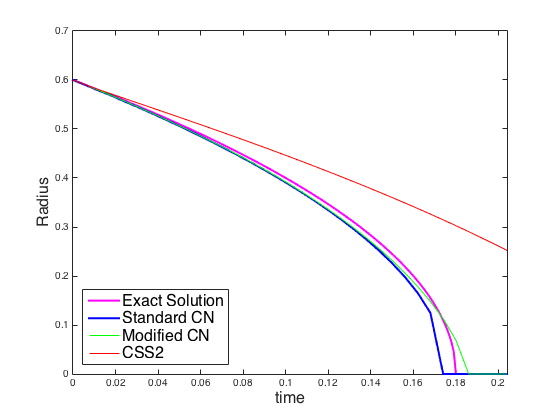}
  \label{fig:2nd-CSS-large}
}
\caption{Different second-order schemes for Allen-Cahn:
  $\epsilon=0.02$, $h=0.015$, and $T=0.17$}
\label{fig:2nd-CSS-smooth}
\end{figure}

\subsection{Artificial convexity}
Following \S\ref{subsec:convexity}, the concept of artificial convexity
scheme can also be applied to the wildly used CSS2 \eqref{eq:CSS2} by
considering the following modified model:
\begin{equation} \label{ConvexAC_2nd}
u_t + \frac{\delta_n}{\epsilon^2}u_{tt} - \Delta u +
\frac{1}{\epsilon^2}f(u) = 0.
\end{equation}
The modified Crank-Nicolson of \eqref{ConvexAC_2nd} can be written as 
$$ 
\begin{aligned}
(\frac{u_h^n-u_h^{n-1}}{k_n}, v_h) & +
(\frac{\delta_n}{\epsilon^2}\cdot\frac{u_h^n - 2u_h^{n-1} +
 u_h^{n-2}}{k_n^2}, v_h) \\ 
& + (\frac{\nabla u_h^n + \nabla
   u_h^{n-1}}{2}, \nabla v_h) + \frac{1}{\epsilon^2}(\tilde{F}[u_h^n,
       u_h^{n-1}], v_h) = 0 \qquad \forall v_h \in V_h,
\end{aligned}
$$ 
which is exactly the CSS2 scheme \eqref{eq:CSS2} when $\delta_n =
\frac{k_n^2}{2}$.

\section{Concluding remarks}
In this paper, we mainly focus on how the behavior of numerical
schemes depend on the time-step size.  For a given finite element
mesh, we compare solutions of fully discrete schemes with moderately
small time step size with those of fully implicit schemes with
extremely small time step size (which can be practically regarded as a
reliable approximation of a semi-discretization scheme).
%It is possible that the solution of semi-discrete scheme is not a good
%approximation to the original phase-field model.
We reach the following conclusions: 

\label{sec:concluding}
\begin{enumerate}
\item A first-order CSS can be mathematically interpreted as a
  standard FIS with a (much) smaller time-step size. As a result, a
  CSS would usually lead to approximation of the solution of the
  original model at a delayed time.  For the Allen-Cahn model, we have
  easily proved this time-delay effect rigorously.  For the
  Cahn-Hilliard model, we observe that, from the numerical
  experiments, CSS also has a similar time-delay effect.  This seems
  to indicate that the solution of the regularized model
  \eqref{ConvexCH} will probably have a time-delay effect in
  comparison to the solution of the original Cahn-Hilliard model
  \eqref{eq:CH}.

\item Since CSS is really an FIS scheme in disguise (at least for the
  cases we have studied in this paper), the value of other FIS should
  not be under-estimated.  Thus a modified FIS is proposed so that the
  maximum principle holds on the discrete level and, as a result, a
  Poisson-like preconditioner can be devised and rigorously analyzed.

%%\item First-order CSS has serious delay and should not be used in
%%  practice

\item A major advantage of any partially implicit scheme is that a
  relatively large time-step size can be used; but such schemes with
  a large time-step size may have time delay (see Figure
\ref{fig:2nd-CSS-large}) and hence may be inaccurate.

%\begin{enumerate}
% \item All schemes are convex when $k$ is sufficiently small, which
% guarantees the energy-stability.  Therefore, when $k$ is sufficiently small, a partially implicit scheme might not be preferable over a fully implicit scheme.
% Now that a fully implicit scheme
% becomes convex when $k$ is sufficiently small, there is no-need to use
% any partially implicit scheme in this case.

\item By using energy minimization we can remove the constraint on the
time step for fully implicit schemes without creating any delayed in
the solutions.

\item Through numerical experiments with modified Crank-Nicolson
scheme, we showed that energy stable is not a sufficient condition
(see Figure \ref{fig:ck-global-local-minimizer}).
That is, an unconditionally stable scheme is not necessarily better
than a conditionally stable scheme. 

\item {\it In summary}, we recommend to use FIS with energy
  minimization.
\end{enumerate}

While most partially implicit schemes have been developed as a numerical
technique for solving a given phase-field model, given the insight
obtained in this paper, we would like to argue that it may be helpful
to view the convex splitting technique as a discrete modeling
technique, namely a procedure to convexify the original model. The
convexified models are \eqref{ConvexAC-delta} and
\eqref{ConvexCH-delta} for the Allen-Cahn Model and the Cahn-Hilliard
equation, respectively.  While neither \eqref{ConvexAC-delta} nor
\eqref{ConvexCH-delta} has a corresponding convexity property on the
continuous level, their appropriately discretized model would have the
desired ``uniform convexity'' properties as stated in Theorem
\ref{thm:convexification}. 

Partially implicit schemes (especially CSS) have been used for many
different models that are different from or more complicated than both
the Allen-Cahn and Cahn-Hilliard equations.  We have not studied
carefully how these schemes behave in those models, but hopefully our
findings in this paper on partially implicit schemes for both the
Allen-Cahn and Cahn-Hilliard models will give some new insight into
the nature of convex splitting technique.

In terms of the unconditional energy-stability, we presented an energy
minimization version of the fully implicit schemes for phase field
modeling.  Although it is challenging to find the global minimizer,
hopefully our findings in this paper on fully implicit schemes for
both the Allen-Cahn and Cahn-Hilliard equations will give some new
insight on the phase field modeling. Accordingly, the design of a fast
solver for the energy minimization problem arising from the phase
field modeling is a research topic of great theoretical and practical
importance.

\section*{References}
\bibliographystyle{elsarticle-num}

\end{document}